\documentclass[11pt, reqno, psamsfonts]{amsart}

\pdfoutput=1
\usepackage{multirow}
\usepackage{tcolorbox}
\usepackage{graphics}
\usepackage[normalem]{ulem}
\usepackage{soul}
\usepackage{comment}

\usepackage{amssymb}
\usepackage{amsthm}
\usepackage{amsmath}
\usepackage{latexsym}
\usepackage[T1]{fontenc}
\usepackage[utf8]{inputenc}
\usepackage[russian, french, english]{babel}

\usepackage{graphicx}
\usepackage{wrapfig, subcaption}
\usepackage[justification=centering, labelfont=bf]{caption}
\usepackage{mathtools}
\usepackage{tikz}
\usepackage{amsbsy}
\usepackage{enumitem}
\usepackage{setspace}
\usepackage{mathrsfs}
\usetikzlibrary{shapes,snakes}
\usetikzlibrary{arrows.meta}
\usetikzlibrary{decorations.pathmorphing}
\usetikzlibrary{patterns}
\usepackage{float}
\usepackage{array}
\usepackage{multicol}
\usepackage{stmaryrd}
\usepackage{cancel} 
\usepackage{lmodern}
\usepackage{algpseudocode}
\usepackage{algorithm}
\usepackage{upgreek}
\usepackage{titlesec}
\usepackage{bbm}
\usepackage[spacing=true,kerning=true,babel=true,tracking=true]{microtype}
\usepackage[shortcuts]{extdash}
\usepackage[foot]{amsaddr}
\usepackage[left=2cm,right=2cm,top=2cm,bottom=2cm,bindingoffset=0cm]{geometry}
\usepackage{bm}
\usepackage{centernot}
\usepackage{mdframed}
\usepackage{multirow}
\usepackage{hyperref}
\hypersetup{
    colorlinks=true,
    linkcolor=blue,
    filecolor=magenta,
    urlcolor=blue,
    citecolor=gray,
}
\allowdisplaybreaks

\usepackage{nameref}
\usepackage{color}

\makeatletter
\let\orgdescriptionlabel\descriptionlabel
\renewcommand*{\descriptionlabel}[1]{%
  \let\orglabel\label
  \let\label\@gobble
  \phantomsection
  \edef\@currentlabel{#1}%
  \let\label\orglabel
  \orgdescriptionlabel{#1}%
}
\makeatother

\makeatletter
\newenvironment{breakablealgorithm}
  {
   \begin{center}
     \refstepcounter{algorithm}
     \hrule height.8pt depth0pt \kern2pt
     \renewcommand{\caption}[2][\relax]{
       {\raggedright\textbf{\ALG@name~\thealgorithm} ##2\par}%
       \ifx\relax##1\relax 
         \addcontentsline{loa}{algorithm}{\protect\numberline{\thealgorithm}##2}%
       \else 
         \addcontentsline{loa}{algorithm}{\protect\numberline{\thealgorithm}##1}%
       \fi
       \kern2pt\hrule\kern2pt
     }
  }{
     \kern2pt\hrule\relax
   \end{center}
  }
\makeatother

\usepackage[backend=biber, style=alphabetic, firstinits=true, sorting=nyt, maxnames=100,backref=true]{biblatex}
\DeclareSourcemap{
  \maps[datatype=bibtex, overwrite]{
    \map{
      \step[fieldset=DOI, null]
      \step[fieldset=URL, null]
      \step[fieldset=ISSN, null]
      \step[fieldset=primaryclass, null]
    }
  }
}

\DeclareNameAlias{author}{last-first}

\title{Coloring Locally Sparse Graphs}
\date{}
\author{\lsstyle James~Anderson$^\dagger$}
\email{james.anderson@math.gatech.edu}
\author{\lsstyle Abhishek~Dhawan$^\star$}
\email{adhawan2@illinois.edu}
\author{\lsstyle Aiya~Kuchukova$^\dagger$}
\email{aiya.kuchukova@math.gatech.edu}
\address{\textls{\normalfont{}$^\dagger$School of Mathematics, Georgia Institute of Technology, Atlanta, GA, USA}}
\address{\normalfont{}$^\star$\textls{Department of Mathematics, University of Illinois Urbana-Champaign, Urbana, IL, USA}}
\thanks{J.A.'s research was partially supported by NSF CAREER grant DMS-2239187 (PI: Anton Bernshteyn). \newline A.D.'s research was partially supported by the Georgia Tech ARC-ACO Fellowship, NSF grant DMS-2053333 (PI: Cheng Mao), NSF CAREER grant DMS-2239187 (PI: Anton Bernshteyn), and NSF RTG grant DMS-1937241. \newline A.K.'s research was partially supported by the Georgia Tech ARC-ACO Fellowship.}

\newtheoremstyle{bfnote}%
{}{}%
{\slshape}{}%
{\bfseries}{\bfseries.}%
{ }%
{\thmname{#1}\thmnumber{ #2}\thmnote{ \ep{\normalfont{}#3}}}

\newtheoremstyle{Claim}%
{}{}%
{\slshape}{}%
{\itshape}{.}%
{ }%
{\thmname{#1}\thmnumber{ #2}\thmnote{ \ep{\normalfont{}#3}}}

\theoremstyle{bfnote}
\newtheorem{theo}{Theorem}[section]
\newtheorem*{theo*}{Theorem}
\newtheorem{prop}[theo]{Proposition}
\newtheorem*{prop*}{Proposition}
\newtheorem{Lemma}[theo]{Lemma}

\newtheorem{corl}[theo]{Corollary}
\newtheorem{conj}[theo]{Conjecture}
\newtheorem{fact}[theo]{Fact}

\newtheorem*{corl*}{Corollary}

\theoremstyle{definition}
\newtheorem{defn}[theo]{Definition}
\newtheorem*{defn*}{Definition}

\newtheorem*{exmp*}{Example}

\theoremstyle{remark}
\newtheorem*{ques*}{Question}
\newtheorem*{remk*}{Remark}

\theoremstyle{Claim}
\newtheorem{claim}{Claim}[theo]

\makeatletter
\newcommand{\neutralize}[1]{\expandafter\let\csname c@#1\endcsname\count@}
\makeatother

\newenvironment{claimproof}{\noindent$\rhd$\hspace{1em}}{\hfill$\blacktriangleleft$\smallskip}

\newcommand{\0}{\emptyset}
\newcommand{\dlll}{\mathfrak{d}}
\newcommand{\set}[1]{\{#1\}}
\newcommand{\N}{{\mathbb{N}}}

\newcommand{\R}{\mathbb{R}}

\renewcommand{\P}{\mathbb{P}}
\newcommand{\E}{\mathbb{E}}
\renewcommand{\epsilon}{\varepsilon}
\newcommand{\eps}{\epsilon}
\renewcommand{\phi}{\varphi}
\renewcommand{\theta}{\vartheta}
\renewcommand{\leq}{\leqslant}
\renewcommand{\geq}{\geqslant}
\newcommand{\defeq}{\coloneqq}
\newcommand{\im}{\mathsf{im}}
\newcommand{\bemph}[1]{{\normalfont#1}} 
\newcommand{\ep}[1]{\bemph{(}#1\bemph{)}} 

\newcommand{\pto}{\dashrightarrow}
\newcommand{\emphdef}[1]{\textbf{\textit{{#1}}}}
\newcommand{\keep}{\mathsf{keep}}
\newcommand{\uncolor}{\mathsf{uncolor}}
\newcommand{\blank}{\mathsf{blank}}

\newcommand{\dom}{\mathsf{dom}}
\newcommand{\eq}{\mathsf{eq}}
\newcommand{\Ber}{\mathsf{Bernoulli}}
\newcommand{\LLL}{\text{Lov\'asz Local Lemma}}
\numberwithin{equation}{section}
\newcommand{\emphd}[1]{\emphdef{#1}}
\newcommand{\Bad}{\mathsf{Bad}(c)}
\newcommand{\Good}{\mathsf{Good}(c)}
\newcommand{\keptedges}{E_{K}(v)}
\newcommand{\uncoloredges}{E_{U}(v)}
\newcommand{\normalizeddeg}{d_{K\cap U}(v)}
\newcommand{\nd}{\normalizeddeg}

\newcommand{\symb}{\#}

\newcommand{\keptEdges}{|\keptedges|}

\titleformat{\section}[block]{\scshape\filcenter}{\thesection.}{1ex}{}
\titleformat{\subsection}[block]{\bfseries}{\thesubsection.}{1ex}{}
\titleformat{\subsection}[block]{\bfseries}{\thesubsection.}{1ex}{}
\titleformat{\subsubsection}[runin]{\itshape}{\bfseries\upshape\thesubsubsection.}{1ex}{}[.---]

\titlespacing*{\section}{0pt}{*3}{*1}
\titlespacing*{\subsection}{0pt}{*3}{*1}
\titlespacing*{\subsubsection}{0pt}{*1.5}{*0}

\renewbibmacro{in:}{}
\renewbibmacro*{volume+number+eid}{%
	\printfield{volume}%
	\setunit*{\addnbspace}
	\printfield{number}%
	\setunit{\addcomma\space}%
	\printfield{eid}}

\DeclareFieldFormat[article]{volume}{\textbf{#1}\space}
\DeclareFieldFormat[article]{number}{\mkbibparens{#1}}

\DeclareFieldFormat{journaltitle}{#1,}
\DeclareFieldFormat[thesis]{title}{\mkbibemph{#1}\addperiod}
\DeclareFieldFormat[article, unpublished, thesis]{title}{\mkbibemph{#1},}
\DeclareFieldFormat[book]{title}{\mkbibemph{#1}\addperiod}
\DeclareFieldFormat[unpublished]{howpublished}{#1, }

\DeclareFieldFormat{pages}{#1}

\DeclareFieldFormat[article]{series}{Ser.~#1\addcomma}

\setlist{topsep=4pt,itemsep=4pt}

\begin{document}


\maketitle

\sloppy

\begin{abstract}
    A graph $G$ is \textit{$k$-locally sparse} if for each vertex $v \in V(G)$, the subgraph induced by its neighborhood contains at most $k$ edges. Alon, Krivelevich, and Sudakov showed that for $f > 0$ if a graph $G$ of maximum degree $\Delta$ is $\Delta^2/f$-locally-sparse, then $\chi(G) = O\left(\Delta/\log f\right)$. We introduce a more general notion of local sparsity by defining graphs $G$ to be \textit{$(k, F)$-locally-sparse} for some graph $F$ if for each vertex $v \in V(G)$ the subgraph induced by the neighborhood of $v$ contains at most $k$ copies of $F$. Employing the R\"{o}dl nibble method, we prove the following generalization of the above result: for every bipartite graph $F$, if $G$ is $(k, F)$-locally-sparse, then $\chi(G) = O\left( \Delta /\log\left(\Delta k^{-1/|V(F)|}\right)\right)$. This improves upon results of Davies, Kang, Pirot, and Sereni who consider the case when $F$ is a path. Our results also recover the best known bound on $\chi(G)$ when $G$ is $K_{1, t, t}$-free for $t \geq 4$, and hold for list and correspondence coloring in the more general so-called ``color-degree'' setting.
\end{abstract}

\section{Introduction}\label{sec: intro}

All graphs considered are finite, undirected, and simple.
A \textit{coloring} of a graph $G = (V(G), E(G))$ is a function $\phi \,: \,V(G) \to \N$, and it is \textit{proper} if $\phi(x) \neq \phi(y)$ whenever $xy \in E(G)$. The \textit{chromatic number} of a graph $G$, denoted $\chi(G)$, is the smallest size of a set $C \subseteq \N$ such that $G$ has a proper coloring $\phi \,:\, V(G) \to C$. Determining $\chi(G)$ for various classes of graphs has been a central topic in graph theory. In this paper, we generalize and extend several results which bound $\chi(G)$ when $G$ satisfies various local sparsity conditions.

\subsection{Preliminaries}

For $k \in \R$, we say $G$ is \emphd{$k$-locally-sparse} if for each $v \in V(G)$, the subgraph induced by $N(v)$ contains at most $\lfloor k\rfloor$ edges (letting $k \in \R$ as opposed to $k \in \N$ permits us to state results more cleanly by allowing us to avoid taking floors). Molloy and Reed showed that if a graph $G$ is locally sparse, then $\chi(G)$ can be bounded from above as follows:

\begin{theo}[{\cite[Lemma 2]{molloy1997bound}}]\label{theo:MR}
    For every $\delta > 0$, there is $\eps > 0$ such that for $\Delta \in \N$ sufficiently large (in terms of $\delta$ and $\epsilon$), if $G$ is a $(1-\delta)\binom{\Delta}{2}$-locally-sparse graph of maximum degree $\Delta$, then $\chi(G) \leq (1- \eps)(\Delta + 1)$. 
\end{theo}

Alon, Krivelevich, and Sudakov generalized this result, giving an asymptotic improvement when the parameter $f$ below is large (for example, when $f = \Theta(\Delta)$\footnote{Throughout this work, we use the standard asymptotic notation $O(\cdot)$, $\Omega(\cdot)$, $o(\cdot)$, etc.}).

\begin{theo}[{\cite[Theorem 1.1]{AKSConjecture}}]\label{theo:AKS}
    There exists a constant $C > 0$ such that for $\Delta \in \N$ sufficiently large and all $f > 1$, if $G$ is a $\frac{\Delta^2}{f}$-locally-sparse graph of maximum degree $\Delta$, then $\chi(G) \leq C\,\Delta/\log f$. 
\end{theo}

In the same paper, they show the above is tight up to the value of $C$ \cite[Proposition 1.2]{AKSConjecture}.
Davies, Kang, Pirot, and Sereni proved Theorem~\ref{theo:AKS} holds with $C = 1 /2 + o(1)$, which remains the best known bound for this problem \cite[Theorem~5]{DKPS}; we note that this result holds in the more general setting of correspondence coloring (see \S\ref{sec: correspondence} for a description of this setting).

For graphs $F$ and $G$, a \emphd{copy} of $F$ in $G$ is a subgraph $H \subseteq G$ (not necessarily induced) which is isomorphic to $F$.
We say a graph $G$ is \emphdef{$F$-free} if $G$ contains no copies of $F$.
As a corollary to Theorem~\ref{theo:AKS}, Alon, Krivelevich, and Sudakov showed the following, where for $t \geq 1$, $K_{1, t, t}$ is the complete triparite graph with set sizes $1,\, t,\, t$:

\begin{corl}[\cite{AKSConjecture}]\label{corl: AKS}
    Let $t \geq 1$. Fix an arbitrary graph $F \subseteq K_{1, t, t}$. Then there exists a real value $c_F>0$ such that if $G$ is an $F$-free graph of maximum degree $\Delta$, then 
    \[\chi(G) \leq(c_F + o_{\Delta}(1))\frac{\Delta}{\log\Delta}.\]
\end{corl}

An outline of their proof goes as follows: if $F \subseteq K_{1, t, t}$ and $G$ is $F$-free, then the neighborhood of each vertex in $G$ does not contain $K_{t, t}$ as a subgraph. The celebrated \hyperref[theo:KST]{K\H{o}v\'ari--S\'os--Tur\'an Theorem} below then implies the neighborhood of each vertex contains relatively few edges:

\begin{theo}[K\H{o}v\'ari--Sós--Turán \cite{KST}; see also Hylt\'{e}n-Cavallius \cite{KST2}]\label{theo:KST}
    Let $G$ be a bipartite graph with a bipartition $V(G) = X \sqcup Y$, where $|X| = m$, $|Y| = n$, and $m \geq n$. Suppose that $G$ does not contain a complete bipartite subgraph with $s$ vertices in $X$ and $t$ vertices in $Y$.
    Then
    \[|E(G)| \leq s^{1/t} m^{1-1/t} n + tm.\]
\end{theo}

Thus when $G$ is $F$-free, $G$ is $k$-locally-sparse for $k = \Theta_t\left(\Delta^{2 - 1/t}\right)$, allowing Theorem~\ref{theo:AKS} to be applied, completing the proof.
This result led Alon, Krivelevich, and Sudakov to conjecture that Corollary~\ref{corl: AKS} holds for all graphs $F$, not just $F \subseteq K_{1, t, t}$.

\begin{conj}[{\cite[Conjecture 3.1]{AKSConjecture}}]\label{conj:AKS}
        Fix an arbitrary graph $F$. Then there exists a real value $c_F>0$ such that if $G$ is  $F$-free and has maximum degree $\Delta$, then 
        \[\chi(G) \leq(c_F + o_{\Delta}(1))\frac{\Delta}{\log\Delta}.\]
\end{conj}

This conjecture has spurred a large research effort; we summarize the progress toward it in \S\ref{sec: history}.
In this paper, we introduce a more general notion of local sparsity:

\begin{defn}\label{def:local_sparse}
    Let $F$ be a graph and let $k \in \R$. 
    A graph $G$ is \emphd{$(k,F)$-sparse} if $G$ contains at most $\lfloor k \rfloor$ copies of $F$ (not necessarily vertex-disjoint).
    A graph $G$ is \emphd{$(k,F)$-locally-sparse} if, for every $v \in V(G)$, the induced subgraph $G[N(v)]$ is $(k, F)$-sparse.
\end{defn}

As a graph $G$ is $k$-locally-sparse if and only if it is $(k, K_2)$-locally-sparse, this definition generalizes the usual notion of local sparsity. The case of $(k, P_t)$-locally-sparse graphs (where $P_t$ is the path on $t$ vertices) was investigated by Davies, Kang, Pirot, and Sereni.
While this result was not stated explicitly in their paper, they discuss the bound in section 4.

\begin{theo}[{\cite{davies2020algorithmic}}]\label{theo:DKPS}
    For every $t \geq 2,\, \eps > 0$, there exist $\Delta_0 \in \N$ and $C > 0$ such that whenever $\Delta \geq \Delta_0$, the following holds:
    Let $k \geq 1/2$ and let $G$ be a $(k, P_t)$-locally-sparse graph of maximum degree $\Delta$. If $k \leq \Delta^2/C$, then \[\chi(G) \leq (1+\eps)\frac{\Delta}{\log(\Delta/\sqrt{k})}.\]
\end{theo}
By taking $t =2$, this bound matches their earlier result \cite[Theorem 5]{DKPS} (which improves upon the constant factor of Theorem~\ref{theo:AKS}).
However, for constant $t > 2$, the neighborhood of a vertex of a graph with maximum degree $\Delta$ can contain $\Theta_t\left(\Delta^t\right)$ copies of $P_t$. 
As $\Delta^2 \ll \Delta^t$ for large $t$, there is wide room to improve the range of $k$ in Theorem~\ref{theo:DKPS}.

Our main theorem improves upon Theorem~\ref{theo:DKPS} in three ways: first, we consider all bipartite graphs $F$ rather than just paths $P_t$; second, we consider the range $k \leq \Delta^{|V(F)|/10}$, which is an improvement for $|V(F)| \geq 20$; finally, we provide an asymptotic improvement on the dependence of $\chi(G)$ on $k$. This improvement is stated below:

\begin{theo}\label{theo: main generalization}
    For every $\eps > 0$ and every bipartite graph $F$, the following holds for $\Delta$ large enough (in terms of $\epsilon$ and $F$).
    Let $1/2 \leq k \leq \Delta^{|V(F)|/10}$ and let $C$ be defined as follows:
    \[C \defeq \left\{\begin{array}{cc}
          4 + \eps, & \text{if } k \leq \Delta^{\eps|V(F)|/200}, \\
          8, & \text{otherwise.}
        \end{array}\right.\]
    Let $G$ be a $(k, F)$-locally-sparse graph of maximum degree $\Delta$.
    Then,
    \[\chi(G) \leq \frac{C\,\Delta}{\log\left(\Delta\,k^{-1/|V(F)|}\right)}.\] 
\end{theo}

We note that Theorem~\ref{theo: main generalization} holds in the more general so-called ``color-degree'' setting for \textit{list coloring} and \textit{correspondence coloring}, which is new even for $F = K_2$; see Theorem~\ref{theo:main_theo}.

\subsection{List coloring, correspondence coloring, and the color-degree setting}\label{sec: correspondence}

Introduced independently by Vizing \cite{vizing1976coloring} and Erd\H{o}s, Rubin, and Taylor \cite{erdos1979choosability}, \textit{list coloring} is a generalization of graph coloring in which each vertex is assigned a color from its own predetermined list of colors. Formally, $L : V(G) \to 2^{\N}$ is a \textit{list assignment} for $G$, and an \textit{$L$-coloring} of $G$ is a proper coloring $\phi: V(G) \to \N$ such that $\phi(v) \in L(v)$ for each $v \in V(G)$. When $|L(v)| \geq q$ for each $v \in V(G)$, where $q \in \N$, we say $L$ is \textit{$q$-fold}. The \textit{list chromatic number} of $G$, denoted $\chi_{\ell}(G)$, is the smallest $q$ such that $G$ has an $L$-coloring for every $q$-fold list assignment $L$ for $G$.

It is often convenient to view list coloring from a different perspective. Given a graph $G$ and a list assignment $L$ for $G$, we create an auxiliary graph $H$ as follows: 
\begin{align*}
    V(H) &\defeq \{(v, c) \in V(G) \times \N : c \in L(v)\} \\
    E(H) &\defeq \{\{(v, c), (u, d)\} : vu \in E(G),\, c = d\}\}.
\end{align*}
We call $H$ a \textit{cover graph} of $G$, and the pair $(L, H)$ a \textit{list cover} of $G$. An $L$-coloring of $G$ is then an independent set $I$ in $H$ which satisfies $|I| = |V(G)|$, and, for each vertex $v \in V(G)$, there exists $c \in \N$ such that $(v, c) \in I$. Thus $I$ selects exactly one vertex of the form $(v,c)$ for each $v \in V(G)$.

\emphd{Correspondence coloring} (also known as \emphd{DP-coloring}) is a generalization of list coloring introduced by Dvo\v{r}\'ak and Postle \cite{DPCol} in order to solve a question of 
Borodin. 
Just as in list coloring, each vertex is assigned a list of colors, $L(v)$;
in contrast to list coloring, though, the identifications between the colors in the lists are allowed to vary from edge to edge.
That is, each edge $uv \in E(G)$ is assigned a matching $M_{uv}$ (not necessarily perfect and possibly empty) from $L(u)$ to $L(v)$.
A proper correspondence coloring is a mapping $\phi : V(G) \to \N$ satisfying $\phi(v) \in L(v)$ for each $v \in V(G)$ and $\phi(u)\phi(v) \notin M_{uv}$ for each $uv \in E(G)$.
Formally, correspondence colorings are defined in terms of an auxiliary graph known as a \emphd{correspondence cover} of $G$.

\begin{defn}[Correspondence Cover]\label{def:corr_cov}
    A \emphd{correspondence cover} of a graph $G$ is a pair $\mathcal{H} = (L, H)$, where $H$ is a graph and $L \,:\,V(G) \to 2^{V(H)}$ such that:
    \begin{enumerate}[label= \ep{\normalfont CC\arabic*}, leftmargin = \leftmargin + 1\parindent]
        \item The set $\set{L(v)\,:\, v\in V(G)}$ forms a partition of $V(H)$,
        \item\label{dp:list_independent} For each $v \in V(G)$, $L(v)$ is an independent set in $H$, and
        \item\label{dp:matching} For each $u, v\in V(G)$, the edge set of $H[L(u) \cup L(v)]$ forms a matching, which is empty if $uv \notin E(G)$.
    \end{enumerate}
\end{defn}
We call the vertices of $H$ \emphd{colors}.
For $c \in V(H)$, we let $L^{-1}(c)$ denote the \emphd{underlying vertex} of $c$ in $G$, i.e., the unique vertex $v \in V(G)$ such that $c \in L(v)$. If two colors $c$, $c' \in V(H)$ are adjacent in $H$, we say that they \emphd{correspond} to each other and write $c \sim c'$.

An \emphd{$\mathcal{H}$-coloring} is a mapping $\phi \colon V(G) \to V(H)$ such that $\phi(v) \in L(v)$ for all $v \in V(G)$. Similarly, a \emphd{partial $\mathcal{H}$-coloring} is a partial mapping $\phi \colon V(G) \pto V(H)$ such that $\phi(v) \in L(v)$ whenever $\phi(v)$ is defined. A \ep{partial} $\mathcal{H}$-coloring $\phi$ is \emphd{proper} if the image of $\phi$ is an independent set in $H$, i.e., if $\phi(u) \not \sim \phi(v)$
for all $u$, $v \in V(G)$ such that $\phi(u)$ and $\phi(v)$ are both defined. Notice, then, that the image of a proper $\mathcal{H}$-coloring of $G$ is exactly an independent set $I \subseteq V(H)$ with $|I \cap L(v)| = 1$ for each $v \in V(G)$.

A correspondence cover $\mathcal{H} = (L,H)$ is \emphdef{$q$-fold} if $|L(v)| \geq q$ for all $v \in V(G)$. The \emphdef{correspondence chromatic number} of $G$, denoted by $\chi_{c}(G)$, is the smallest $q$ such that $G$ admits a proper $\mathcal{H}$-coloring for every $q$-fold correspondence cover $\mathcal{H}$.

Note that a list cover $(L, H)$ of $G$ is a correspondence cover of $G$, where each matching $M_{uv}$ is such that $(u,c)$ matches with $(v,c)$ for each $c \in \N$. As classical coloring is the special case of list coloring in which all lists are identical, it follows that $\chi(G) \leq \chi_{\ell}(G) \leq \chi_{c}(G)$.

When considering locally sparse graphs, a correspondence coloring version of Theorem~\ref{theo:DKPS} was proven by
Davies, Kang, Pirot, and Sereni, albeit for a smaller range of $k$.
Once again, while they do not explicitly state this result, they discuss it in section 4 of \cite{davies2020algorithmic}. 

\begin{theo}[Correspondence Coloring version of Theorem~\ref{theo:DKPS}, {\cite{davies2020algorithmic}}]\label{theo:DKPSCorrespondence}
    For every $t \geq 2,\, \eps > 0$, there exists $\Delta_0 \in \N$ such that whenever $\Delta > \Delta_0$, the following holds:
    let $k \geq 1/2$ and let $G$ be a $(k, P_t)$-locally-sparse graph of maximum degree $\Delta$. If $k \leq \Delta^2/(\log \Delta)^{2/\eps}$, then \[\chi_c(G) \leq (1+\eps)\frac{\Delta}{\log(\Delta/\sqrt{k})}.\]
\end{theo}

A curious feature of correspondence coloring is that structural constraints can be placed on the cover graph $H$ instead of on the underlying graph $G$.
For instance, if $\mathcal{H} = (L, H)$ is a correspondence cover of a graph $G$, then $\Delta(H) \leq \Delta(G)$, so an upper bound on $\Delta(H)$ is a weaker assumption than the same upper bound on $\Delta(G)$, and there exist a number of results in list and correspondence coloring in which the number of available colors given to each vertex is a function of $\Delta(H)$ as opposed to $\Delta(G)$. This framework, often referred to as the \emphdef{color-degree setting}, was pioneered by Kahn \cite{KahnListEdge}, Kim \cite{Kim95}, Johansson \cite{Joh_triangle,Joh_sparse}, and Reed \cite{Reed}, among others. For a selection of a few more recent examples, see \cite{BohmanHolzman, ReedSud, LohSudakov, alon2020palette, cambie2022independent, KangKelly, GlockSudakov, anderson2022coloring, anderson2023colouring} (see, also, \S\ref{subsection: color degree}).
Similarly, as a result of condition \ref{dp:matching}, if $G$ is $(k, F)$-locally-sparse, then so is $H$. However, the same does not hold for $(k, F)$-sparsity. For example, consider a $q$-fold list cover where all the lists are the same. Then for any graph $F$, each copy of $F$ in $G$ corresponds to at least $q$ copies of $F$ in $H$. Thus $G$ may contain just one copy of $F$, while $H$ contains many. 
Even for $F$-free graphs $G$, there can exist a cover $H$ of $G$ which contains $F$. For instance, in Fig.~\ref{fig:not free}, the graph $G$ is $C_6$-free, though its cover contains a $C_6$.
Nevertheless, appropriate conditions exist that ensure whenever $G$ is $F$-free, then so is any cover $H$ of $G$.
We summarize our observations in the following proposition (see \S\ref{sec:appendix} for the proof):

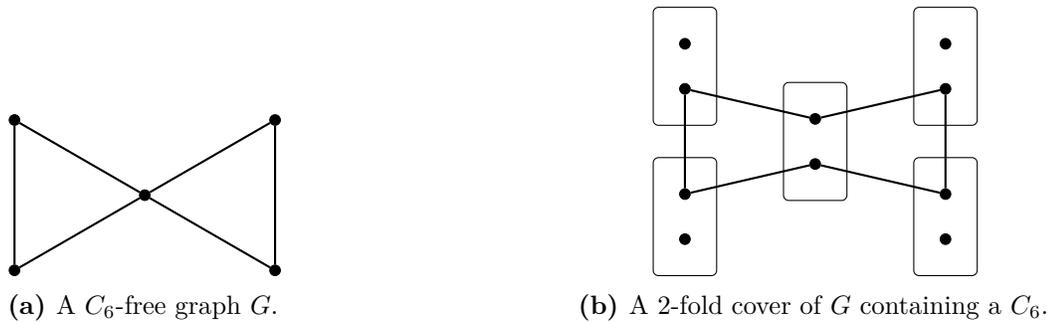
\begin{figure}[htb!]
	\centering
	\begin{subfigure}[t]{.4\textwidth}
		\centering
		\begin{tikzpicture}[scale=2]
		\node[circle,fill=black,draw,inner sep=0pt,minimum size=4pt] (a) at (0,0) {};
		\path (a) ++(90:1) node[circle,fill=black,draw,inner sep=0pt,minimum size=4pt] (b) {};
            \path (a) ++(30:1) node[circle,fill=black,draw,inner sep=0pt,minimum size=4pt] (c) {};
            \path (c) ++(30:1) node[circle,fill=black,draw,inner sep=0pt,minimum size=4pt] (d) {};
            \path (d) ++(-90:1) node[circle,fill=black,draw,inner sep=0pt,minimum size=4pt] (e) {};

            \draw[thick] (c) -- (a) -- (b) -- (c) -- (d) -- (e) -- (c);
		
		\end{tikzpicture}
		\caption{A $C_6$-free graph $G$.}\label{fig:not free:graph}
	\end{subfigure}%
	\qquad\qquad\qquad%
	\begin{subfigure}[t]{.4\textwidth}
		\centering
		\begin{tikzpicture}[scale=2]
		\node[circle,fill=black,draw,inner sep=0pt,minimum size=4pt] (a) at (0,0) {};
		\path (a) ++(90:1) node[circle,fill=black,draw,inner sep=0pt,minimum size=4pt] (b) {};
            \path (a) ++(30:1) node[circle,fill=black,draw,inner sep=0pt,minimum size=4pt] (c) {};
            \path (c) ++(30:1) node[circle,fill=black,draw,inner sep=0pt,minimum size=4pt] (d) {};
            \path (d) ++(-90:1) node[circle,fill=black,draw,inner sep=0pt,minimum size=4pt] (e) {};

            \path (a) ++(90:0.3) node[circle,fill=black,draw,inner sep=0pt,minimum size=4pt] (a1) {};
            \path (b) ++(90:0.3) node[circle,fill=black,draw,inner sep=0pt,minimum size=4pt] (b1) {};
            \path (c) ++(90:0.3) node[circle,fill=black,draw,inner sep=0pt,minimum size=4pt] (c1) {};
            \path (d) ++(90:0.3) node[circle,fill=black,draw,inner sep=0pt,minimum size=4pt] (d1) {};
            \path (e) ++(90:0.3) node[circle,fill=black,draw,inner sep=0pt,minimum size=4pt] (e1) {};

            \draw[rounded corners=2pt, black] ([xshift=-6pt, yshift=8pt]a1.south) -- ([xshift=-6pt, yshift=-8pt]a.north) -- ([xshift=6pt, yshift=-8pt]a.north) -- ([xshift=6pt, yshift=8pt]a1.south) -- cycle;

            \draw[rounded corners=2pt, black] ([xshift=-6pt, yshift=8pt]b1.south) -- ([xshift=-6pt, yshift=-8pt]b.north) -- ([xshift=6pt, yshift=-8pt]b.north) -- ([xshift=6pt, yshift=8pt]b1.south) -- cycle;

            \draw[rounded corners=2pt, black] ([xshift=-6pt, yshift=8pt]c1.south) -- ([xshift=-6pt, yshift=-8pt]c.north) -- ([xshift=6pt, yshift=-8pt]c.north) -- ([xshift=6pt, yshift=8pt]c1.south) -- cycle;

            \draw[rounded corners=2pt, black] ([xshift=-6pt, yshift=8pt]d1.south) -- ([xshift=-6pt, yshift=-8pt]d.north) -- ([xshift=6pt, yshift=-8pt]d.north) -- ([xshift=6pt, yshift=8pt]d1.south) -- cycle;

            \draw[rounded corners=2pt, black] ([xshift=-6pt, yshift=8pt]e1.south) -- ([xshift=-6pt, yshift=-8pt]e.north) -- ([xshift=6pt, yshift=-8pt]e.north) -- ([xshift=6pt, yshift=8pt]e1.south) -- cycle;

            \draw[thick] (a1) -- (c) -- (e1) -- (d) -- (c1) -- (b) -- (a1);

		\end{tikzpicture}
		\caption{A $2$-fold cover of $G$ containing a $C_6$.}\label{fig:not free:cover}
	\end{subfigure}%
	\caption{A $C_6$-free graph with a $2$-fold cover containing a $C_6$.}\label{fig:not free}
\end{figure}

\begin{prop}\label{prop: G F free implies H F free}
    Let $F$ and $G$ be graphs and let $\mathcal{H} = (L, H)$ be a correspondence cover of $G$.
    For $k \in \R$, the following holds:
    \begin{enumerate}[label=\ep{\normalfont S\arabic*}]
        \item\label{item:local} If $G$ is $(k, F)$-locally-sparse, then so is $H$. 
        
        \item\label{item:global} If $F$ satisfies
        \[\forall\, u, v \in V(F),\quad uv \notin E(F) \implies N_F(u) \cap N_F(v) \neq \0,\]
        then if $G$ is $F$-free, $H$ is as well.
    \end{enumerate}
\end{prop}

For any graph $F$, there is a complete $\chi(F)$-partite graph $F'$ such that $F\subseteq F'$ and $|V(F)| = |V(F')|$.
If $G$ is $F$-free, then it is $F'$-free as well.
Note that any complete $r$-partite graph $F$ for $r \geq 2$ satisfies the conditions of \ref{item:global}.
Therefore, we obtain the following corollary:

\begin{corl}\label{corl:F-free}
    Let $G$ be an $F$-free graph for some $F$ satisfying $\chi(F) = q \geq 2$ and let $\phi$ be a proper $q$-coloring of $F$.
    Define $F'$ to be the complete $q$-partite graph with partitions $V_1, \ldots, V_q$ satisfying $|V_i| = |\phi^{-1}(i)|$.
    For any correspondence cover $\mathcal{H} = (L, H)$ of $G$, the cover graph $H$ is $F'$-free.
\end{corl}

\subsection{Main results}\label{subsec:main_results}

We are ready to present our main results.
We begin with a color-degree version of Theorem~\ref{theo: main generalization} for correspondence coloring.

\begin{theo}\label{theo:main_theo}
    There exists a constant $\alpha > 0$ such that for every $\epsilon > 0$, there is $d^* \in \N$ such that the following holds. 
    Suppose that $d$, $s$, $t \in \N$, and $k \in \R$ satisfy
    \[
        d \geq d^*,\quad 1 \leq t \leq s, \quad st \leq \frac{\alpha\,\epsilon\,\log d}{\log \log d}, \quad \text{and} \quad 1/2 \leq k \leq d^{(s+t)/10}.
    \]
    Define $C$ as follows:
    \[C \defeq \left\{\begin{array}{cc}
          4 + \eps &  \text{if } k \leq d^{\eps(s+t)/200}, \\
          8 & \text{otherwise.}
        \end{array}\right.\]
    If $G$ is a graph and $\mathcal{H} = (L,H)$ is a correspondence cover of $G$ such that:
    \begin{enumerate}[label=\ep{\normalfont\roman*}]
        \item $H$ is $(k, K_{s,t})$-locally-sparse,
        \item $\Delta(H) \leq d$, and
        \item $|L(v)| \geq C\,d/\log \left(dk^{-1/(s+t)}\right)$ for all $v \in V(G)$,
    \end{enumerate}
    then $G$ admits a proper $\mathcal{H}$-coloring.
\end{theo}

We note that both the local sparsity and the degree constraints are on the cover graph $H$ as opposed to $G$, which is a weaker assumption as a result of Proposition~\ref{prop: G F free implies H F free}.
This is new even for $s = t = 1$.\footnote{This special case was further investigated in follow-up work of the second named author of this manuscript with a focus on algorithmic implications \cite{dhawan2024palette}.}
We also remark that the upper bound $k \leq d^{(s+t)/10}$ can be relaxed slightly while increasing the value $C$, however, our approach does not work for $k = \Theta(d^{s+t})$.
We discuss this further in \S\ref{sec: proof overview}.

Note that for any bipartite graph $F$, a $(k, F)$-locally-sparse graph is $(k, K_{s, t})$-locally-sparse for appropriate $s, t \in \N$ satisfying $s+t = |V(F)|$.
Therefore, by treating $s$ and $t$ as constants we immediately obtain the following corollary:

\begin{corl}\label{corl:bipartite}
    For every $\eps > 0$ and bipartite graph $F$, the following holds for $d$ large enough.
    Let $1/2 \leq k \leq d^{|V(F)|/10}$ and let $C$ be defined as follows:
    \[C \defeq \left\{\begin{array}{cc}
          4 + \eps &  \text{if } k \leq d^{\eps|V(F)|/200}, \\
          8 & \text{otherwise.}
        \end{array}\right.\]
    Let $G$ be a graph and $\mathcal{H} = (L, H)$ be a correspondence cover of $G$ such that:
    \begin{enumerate}
        \item $H$ is $(k, F)$-locally-sparse, 
        \item $\Delta(H) \leq d$, and
        \item $|L(v)| \geq C\,d/\log\left(dk^{-1/|V(F)|}\right)$ for all $v \in V(G)$,
    \end{enumerate}
    then $G$ admits a proper $\mathcal{H}$-coloring.
\end{corl}

Setting $d = \Delta(G)$ together with Proposition~\ref{prop: G F free implies H F free} yields the following corollary:

\begin{corl}\label{corl:DP_chromatic_number}
    For every $\eps > 0$ and bipartite graph $F$, the following holds for $\Delta$ large enough.
    Let $1/2 \leq k \leq \Delta^{|V(F)|/10}$ and let $C$ be defined as follows:
    \[C \defeq \left\{\begin{array}{cc}
          4 + \eps, &   \text{if }k \leq \Delta^{\eps|V(F)|/200}, \\
          8, & \text{otherwise.}
        \end{array}\right.\]
    Let $G$ be a $(k, F)$-locally-sparse graph of maximum degree $\Delta$.
    Then,
    \[\chi_c(G) \leq \frac{C\,\Delta}{\log\left(\Delta k^{-1/|V(F)|}\right)}.\]
\end{corl}

As $\chi(G) \leq \chi_c(G)$, this clearly implies Theorem~\ref{theo: main generalization}, and, for $F = P_2$, recovers Theorem~\ref{theo:DKPSCorrespondence} up to a constant (albeit with a stricter bound on $k$). More importantly, for $t \geq 3$, our results asymptotically improve the bound in Theorem~\ref{theo:DKPSCorrespondence} (i.e., $k^{-1/t}$ as opposed to $k^{-1/2}$). Observe as well that for $t\geq 20$, our results hold for an asymptotically larger range of $k$ than Theorem~\ref{theo:DKPSCorrespondence}.

We conclude this section with a summary of the results contained in this manuscript. Primarily, we introduce a generalized notion of local sparsity, allowing us to rephrase the results of Alon, Krivelevich, and Sudakov \cite{AKSConjecture} and Davies, Kang, Pirot, and Sereni \cite{DKPS} as theorems about $(k, F)$-locally-sparse graphs for when $F$ is an edge or a path, respectively. Our theorem considers arbitrary bipartite graphs $F$, thereby recovering and improving upon both of these results. In addition, we consider the more general setting of placing sparsity and degree constraints on the cover graph $H$ rather than on the underlying graph $G$. This is the first color-degree version of Theorem~\ref{theo:AKS}, even for list coloring. We also recover a result of \cite{anderson2022coloring} (which we discuss further in subsequent sections; see Theorem~\ref{theo:our_old_result}), which to date is the best known bound on $\chi(G)$ for the largest class of graphs known to satisfy Conjecture~\ref{conj:AKS}.

\subsubsection*{Structure of the paper}

The rest of the paper is structured as follows.
In \S\ref{section: relation to prior work}, we discuss related works in the graph coloring literature to better place our results in context.
In \S\ref{sec: proof overview}, we provide an informal overview of our proof techniques.
In \S\ref{sec:wcp}, we will formally describe a coloring procedure we employ in our proof and prove a key lemma regarding its output.
In \S\ref{sec:main_proof}, we apply this key lemma iteratively to prove Theorem~\ref{theo:main_theo}.
Finally, in \S\ref{subsec:open_probs}, we discuss potential avenues for future research.

\section{Discussion of previous work}\label{section: relation to prior work}

In this section, we discuss related works in the graph coloring literature.
Specifically, we survey results related to Conjecture~\ref{conj:AKS} and the color-degree setting, and compare our proof techniques to other approaches in the area.

\subsection{History of the Alon--Krivelevich--Sudakov conjecture
}\label{sec: history}

A trivial upper bound on $\chi(G)$ comes from a greedy coloring, which shows $\chi(G) \leq \Delta(G) + 1$, where $\Delta(G)$ is the maximum degree of $G$. Brooks improved this bound to $\chi(G) \leq \Delta(G)$ when $\Delta(G) \geq 3$ and $G$ contains no cliques of size $\Delta(G) + 1$ \cite{brooks1941colouring}.
Reed improved upon this, showing that $\chi(G) \leq \Delta(G) - 1$ for sufficiently large $\Delta(G)$ when $G$ contains no cliques of size $\Delta(G)$ \cite{reed1999strengthening}. Conjecture~\ref{conj:AKS} is a natural extension of these results, i.e., can the theorems of Brooks and Reed be improved by forbidding subgraphs $F$ other than cliques? 
Johansson \cite{Joh_sparse} showed that for any graph $F$, if $G$ is $F$-free and has maximum degree $\Delta$, then 
\[\chi(G) = O\left(\frac{\Delta\log\log\Delta}{\log \Delta}\right),\]
where the $O(\cdot)$ is with respect to $\Delta$ and hides constants that may depend on $F$. This result was never published, however, employing a result of Shearer \cite{shearer1995independence}, both Molloy \cite{Molloy} and Bernshteyn \cite{bernshteyn2019johansson} provide a proof of this result in the more general settings of list coloring and correspondence coloring, respectively. In \cite{Joh_triangle},
Johansson removed the $\log\log \Delta$ factor for $F = K_3$ (see Molloy and Reed's book \cite[Chapter 13]{MolloyReed} for a textbook presentation of this proof), and in 1999, Alon, Krivelevich, and Sudakov extended this result to \emphdef{almost-bipartite graphs}, i.e., subgraphs of $K_{1, t, t}$ for some $t$ (see Corollary~\ref{corl: AKS}), leading them to pose Conjecture~\ref{conj:AKS}.

To date, almost-bipartite graphs remain the largest class of connected graphs for which Conjecture~\ref{conj:AKS} is known to hold, and generalizing this result to a larger class of graphs remains a tantalizing open problem. However, as long as $F$ contains a cycle, the bound in Conjecture~\ref{conj:AKS} is best possible up to the value of $c_F$, since there exist $\Delta$-regular graphs $G$ of arbitrarily high girth with $\chi(G) \geq (1/2)\Delta/\log \Delta$ \cite{BollobasIndependence} (when $F$ is acyclic, a simple degeneracy argument shows $\chi(G) = O(1)$). Therefore, much research has focused on improving upper bounds on the value of $c_F$ in an attempt to try to close the gap between the known upper and lower bounds. 
We highlight progress in Table~\ref{table:bounds} for the interested reader. (Note that analogues of all these results hold for \textit{list} and \textit{correspondence} coloring.)

\begin{table}[htb!]
	\resizebox{!}{3cm}{%
 \begin{tabular}{| l || l || l |}
				\hline
				$F$ & $c_F$ & References\\\hline\hline
				forest & $0$ & Follows from degeneracy\\\hline
				not forest & $\geq 1/2$ & Bollob\'as \cite{BollobasIndependence}
				\\\hline
				\multirow{3}{*}{$K_3$} & finite & Johansson \cite{Joh_triangle} \\\cline{2-3}
				& $\leq 4$ & Pettie--Su \cite{PS15} \\\cline{2-3}
				& $\leq 1$ & Molloy \cite{Molloy}\\\hline
				cycle/fan & $\leq 1$ & Davies--Kang--Pirot--Sereni \cite[\S5.2, \S5.5]{DKPS}\\\hline
				bipartite & $\leq 1$ & Anderson--Bernshteyn--Dhawan \cite{anderson2023colouring}\\\hline
				\multirow{3}{*}{$K_{1,t,t}$} & $O(t)$ & Alon--Krivelevich--Sudakov \cite{AKSConjecture} \\\cline{2-3}
				& $\leq t$ & Davies--Kang--Pirot--Sereni \cite[\S5.6]{DKPS}\\\cline{2-3}
				& $\leq 4$ & Anderson--Bernshteyn--Dhawan \cite{anderson2022coloring}\\\hline
			\end{tabular}%
  }
  \caption{Known bounds on $c_F$.}\label{table:bounds}
\end{table}

In particular, Davies, Kang, Pirot, and Sereni proved $c_F \leq 1$ for cycles and fans \cite[\S5.2, \S5.5]{DKPS}, and in the same paper, improved upon the result of Alon, Krivelevich, and Sudakov by showing $c_F \leq t$ for $F = K_{1, t, t}$ \cite[\S5.6]{DKPS}. The first two authors of this paper, together with Bernshteyn, showed $c_F \leq 1$ when $F$ is bipartite \cite{anderson2023colouring}; and, in addition, showed $c_F \leq 4$ when $F \subseteq K_{1, t, t}$, thereby making the bound for $c_F$ independent of $t$ \cite{anderson2022coloring}. This improves the bounds of Davies, Kang, Pirot, and Sereni when $t \geq 5$. We note that when we set $k < 1$, Theorem~\ref{theo: main generalization} recovers the aforementioned result of \cite{anderson2022coloring}.

\subsection{The color-degree setting}\label{subsection: color degree}

Recall from \S\ref{sec: history} the result of Johansson that $\chi(G) = O(\Delta/\log\Delta)$ for $K_3$-free graphs $G$ having maximum degree $\Delta$ sufficiently large.
Alon and Assadi extended Johansson's result in the list coloring setting by placing degree restrictions on the cover graph $H$ as opposed to $G$. 

\begin{theo}[{\cite[Proposition 3.2]{alon2020palette}}]\label{theo:alon_assadi}
    The following holds for $d$ sufficiently large.
    Let $G$ be a $K_3$-free graph and let $\mathcal{H} = (L, H)$ be a list cover of $G$ satisfying the following:
    \[|L(v)| \geq \frac{8\,d}{\log d}, \quad \Delta(H) \leq d.\]
    Then, $G$ admits a proper $\mathcal{H}$-coloring.
\end{theo}

Alon, Cambie, and Kang were able to prove a similar bound in the color-degree setting for bipartite graphs $G$ with an improved constant of $1+o(1)$ as opposed to $8$ \cite{alon2021asymmetric}.
In subsequent work, Cambie and Kang extended these results on bipartite graphs to correspondence coloring, where they show the following:

\begin{theo}[{\cite[Corollary 1.3]{cambie2022independent}}]\label{theo:ck}
    For all $\eps > 0$, the following holds for $d$ sufficiently large.
    Let $G$ be a bipartite graph and let $\mathcal{H} = (L, H)$ be a correspondence cover of $G$ satisfying the following:
    \[|L(v)| \geq (1+\eps)\frac{d}{\log d}, \quad \Delta(H) \leq d.\]
    Then, $G$ admits a proper $\mathcal{H}$-coloring.
\end{theo}

Cambie and Kang conjectured that Theorem~\ref{theo:ck} holds for $K_3$-free graphs as well, i.e., a color-degree version of Bernshteyn's result mentioned earlier.

\begin{conj}[{\cite[Conjecture 1.4]{cambie2022independent}}]\label{conj:CK}
    For all $\eps > 0$, the following holds for $d$ sufficiently large.
    Let $G$ be a $K_3$-free graph and let $\mathcal{H} = (L, H)$ be a correspondence cover of $G$ satisfying the following:
    \[|L(v)| \geq (1+\eps)\frac{d}{\log d}, \quad \Delta(H) \leq d.\]
    Then, $G$ admits a proper $\mathcal{H}$-coloring.
\end{conj}

This is still an open problem, however, the first two authors of this manuscript along with Bernshteyn made progress towards it in \cite{anderson2022coloring}, achieving a bound of $(4+\epsilon)d/\log d$. They also generalized the result from triangle-free graphs $G$ to $F$-free covers $H$ when $F$ is almost-bipartite:

\begin{theo}[{\cite[Corollary 1.11]{anderson2022coloring}}]\label{theo:our_old_result}
    Let $F$ be an almost-bipartite graph.
    For all $\eps > 0$, the following holds for $d$ sufficiently large.
    Let $G$ be a graph and let $\mathcal{H} = (L, H)$ be a correspondence cover of $G$ satisfying the following:
    \begin{itemize}
        \item $H$ is $F$-free,
        \item $\Delta(H) \leq d$, and
        \item $|L(v)| \geq (4+\eps)d/\log d$.
    \end{itemize}
    Then $G$ admits a proper $\mathcal{H}$-coloring.
\end{theo}

As $K_3$ is an almost-bipartite graph, Theorem~\ref{theo:our_old_result} along with Corollary~\ref{corl:F-free}  constitutes progress toward Conjecture~\ref{conj:CK}.
We remark that it is not known if Conjecture~\ref{conj:CK} holds even for list covers.

\subsection{Comparison of proof techniques with prior work}\label{subsection: comparison}

In this section, we compare our techniques to earlier approaches in related works.
A core component of our proof is a result showing that $(k, K_{s, t})$-locally-sparse graphs are $k'$-locally-sparse for appropriate $k'$ (see Proposition~\ref{prop:KST_multiple}).
We prove Theorem~\ref{theo:main_theo} by combining this result with the ``R\"odl nibble method'' for graph coloring, an iterative procedure to construct the desired coloring.
We provide a detailed overview in~\S\ref{sec: proof overview}.

The astute reader may wonder why Proposition~\ref{prop:KST_multiple} together with Theorem~\ref{theo:DKPSCorrespondence} for $t = 2$ cannot be used to prove Corollary~\ref{corl:DP_chromatic_number}. 
The bound obtained by this method would be as follows for $k^\star = \frac{(s+t)^2\,\Delta^{s+t - 1}}{2\,s^s\,t^t}$:
\begin{align}\label{eq: dkps + prop KST multiple}
    \chi_c(G) = \left\{\begin{array}{cc}
    O\left(\dfrac{s\,t\,\Delta}{\log \Delta}\right) &  \text{ if } k \leq k^\star;\vspace{7pt} \\
    O\left(\dfrac{s\,t\,\Delta}{(s+t)\log \left(\Delta\, k^{-1/(s+t)}\right)}\right) &  \text{ if } k > k^\star.
\end{array}\right.
\end{align}
Note that our result falls in the setting where $k \leq k^\star$.
Our approach not only yields an improved dependence on $k$, but also achieves a constant factor independent of $s$ and $t$.
Additionally, our result is significantly stronger in the regime $st = \omega(1)$.
Indeed, while Corollary~\ref{corl:DP_chromatic_number} is not stated as such, one can deduce from Theorem~\ref{theo:main_theo} that the result holds for all $s, t$ satisfying $st = O(\log \Delta / \log \log \Delta)$.
Clearly, Corollary~\ref{corl:DP_chromatic_number} outperforms the bound from \eqref{eq: dkps + prop KST multiple} in this regime.
Finally, we recall that our result yields the first color-degree proof of Theorem~\ref{theo:AKS} as a special case.
In fact, the proof strategies of both Theorem~\ref{theo:AKS} and Theorem~\ref{theo:DKPSCorrespondence} fail in this setting, which is considered the benchmark in the graph coloring literature.

We remark that a similar distinction in approach is evident when comparing the works \cite{DKPS} and \cite{anderson2022coloring} regarding the value of $c_{K_{1, t, t}}$.
In \cite{DKPS}, the authors observe that every $K_{1, t, t}$-free graph satisfies $e(G[N(v)]) = O_t(\Delta^{2-1/t})$ to deduce their bound $c_{K_{1, t, t}} \leq t$ from Theorem~\ref{theo:DKPSCorrespondence}.
In contrast, the first two authors of this manuscript along with Bernshteyn prove the bound $c_{K_{1, t, t}} \leq 4$ in \cite{anderson2022coloring} (see Theorem~\ref{theo:our_old_result}) through a stronger observation: every $K_{1, t, t}$-free graph satisfies $e(G[S]) = O_t(|S|^{2-1/t})$ for any $v\in V(G)$ and $S \subseteq N(v)$.
This allows them to take advantage of the local sparsity through the iterations of the nibble method and improve the bound of $c_{K_{1, t, t}}$ for $t > 4$.

We remark that we employ a similar observation in our proof: every $(k,K_{s, t})$-locally-sparse graph satisfies $G[S]$ is $(k, K_{s, t})$-sparse for any $v\in V(G)$ and $S \subseteq N(v)$.
In fact, our approach is inspired by that of \cite{anderson2022coloring} and we apply a number of their arguments in a black box manner (we discuss this further in~\S\ref{sec: proof overview}).

We conclude this section with a brief discussion of the bipartiteness of $F$.
The version of the nibble method that we employ can be adapted to any $F$-free graph where $F$ satisfies $ex(n, F) = o(n^2)$.
Indeed, by the celebrated Erd\H{o}s--Stone Theorem \cite{erdos1946structure}, bipartite graphs are the only graphs for which this is true.
It is, therefore, unclear how to extend our result to arbitrary graphs $F$.\footnote{In follow-up work, the second named author of this manuscript considers the case $F = K_r$ for $r \geq 3$; the results obtained are far weaker than those in this paper \cite{dhawan2024bounds}.}

\section{Proof overview}\label{sec: proof overview}

Let $G$ be a graph and let $\mathcal{H} = (L, H)$ be a correspondence cover of $G$ satisfying the conditions of Theorem~\ref{theo:main_theo}.
In order to find a proper $\mathcal{H}$-coloring of $G$, we use a variant of the well-known ``R\"odl Nibble method'', in which we randomly color a small fraction of vertices and repeat the procedure on the uncolored vertices.
In particular, we use a coloring procedure developed by Kim \cite{Kim95}, based on a technique of R\"odl in \cite{RODL198569}. A version of Kim's procedure was further developed by Molloy and Reed \cite{MolloyReed}, which they call the \hyperref[algorithm: wcp]{Wasteful Coloring Procedure}.
We specifically employ the version generalized to correspondence coloring by the first two authors and Bernshteyn in \cite{anderson2022coloring} (see \S\ref{subsec: alg overview} for a detailed description).
At each step of the procedure, we start with a graph $G$ and a correspondence cover $\mathcal H$ of $G$, and randomly construct a partial $\mathcal H$-coloring $\phi$ of $G$ and lists $L'(v) \subseteq L(v)$ satisfying certain properties.
In particular, let $G_\phi$ be the graph induced by the uncolored vertices under $\phi$.
For a vertex $v \in V(G_\phi)$ and a color $c \in L(v)$, it is possible $\phi$ cannot be extended by coloring $v$ with $c$, as there may be a vertex $u \in N_G(v)$ such that $c \sim \phi(u)$.
We therefore define $L_\phi(v)$ to be the set of \textit{available colors} for $v$, i.e.,
\[L_\phi(v) \defeq \set{c \in L_\phi(v)\,:\, N_H(c) \cap \im(\phi) = \0}.\]
It seems natural that the lists $L'(\cdot)$ produced by the \hyperref[algorithm: wcp]{Wasteful Coloring Procedure} should satisfy $L'(v) = L_{\phi}(v)$; however, in actuality $L'(v)$ is potentially a strict subset of $L_{\phi}(v)$ (hence the ``wasteful'' in the name; see the discussion in \cite[Chapter 12.2.1]{MolloyReed} for the utility of such ``wastefulness'').

Let us define a cover $\mathcal{H}'$ for the graph $G_\phi$ by letting
\[H' \defeq H\left[\bigcup_{v \in V(G_{\phi})}L'(v)\right].\]
Let $\ell \defeq \min |L(v)|$, $\ell' \defeq \min|L'(v)|$, $d \defeq \Delta(H)$, and $d' \defeq \Delta(H')$.
We say $\phi$ is ``good'' if the ratio $d'/\ell'$ is considerably smaller than $d/\ell$.
If indeed the output of the \hyperref[algorithm: wcp]{Wasteful Coloring Procedure} produces a ``good'' coloring, then we may repeatedly apply the procedure until we are left with a cover $(\tilde L, \tilde H)$ of the uncolored vertices such that $\min |\tilde L(v)| \geq 8\Delta(\tilde H)$.
At this point, we can complete the coloring by applying the following proposition:

\begin{prop}\label{prop:final_blow}
    Let $\mathcal{H} = (L, H)$ be a correspondence cover of $G$. 
    If there is an integer $\ell$ such that $|L(v)| \geq \ell$ and $\deg_H(c) \leq \ell/8$ for each $v \in V(G)$ and $c \in V(H)$, then $G$ admits a proper $\mathcal{H}$-coloring.
\end{prop}
 
This proposition is a rather standard application of the \hyperref[LLL]{Lov\'asz Local Lemma} (Theorem~\ref{LLL} below), and one can find the first use of this proposition (for list colorings) in \cite{reed1999strengthening}, and its correspondence coloring version in \cite[Appendix]{bernshteyn2019johansson}.

Therefore, we spend considerable effort showing that the output of the \hyperref[algorithm: wcp]{Wasteful Coloring Procedure} is ``good''.
The first step in showing this is to compute bounds on the expected values of the list sizes $|L'(v)|$ and degrees of colors $\deg_{H'}(c)$.
In fact, we can show that
\[\frac{\E[\deg_{H'}(c)]}{\E[|L'(v)|]} \approx \uncolor\,\frac{d}{\ell},\]
where $\uncolor < 1$ is a certain quantity defined precisely in \S\ref{subsec: alg overview}.
With this in hand, the next step would be to show that the above expression holds with high probability as we may then apply the \hyperref[LLL]{Lov\'asz Local Lemma} to conclude there exists a ``good'' partial coloring $\phi$.
Unfortunately, the random variable $\deg_{H'}(c)$ may not be concentrated about its expected value if the neighbors of $c$ have many common neighbors.
In particular, the events $\set{c' \in V(H')}$ for $c' \in N_H(c)$ may be strongly correlated, increasing the variance of $\deg_{H'}(c)$.
An idea of Jamall \cite{Jamall}, which was developed further by Pettie and Su \cite{PS15}, Alon and Assadi \cite{alon2020palette}, and the first two authors and Bernshteyn \cite{anderson2022coloring} allows one to overcome this issue.
Specifically, instead of considering the \textit{maximum color-degree} we consider the \textit{average color-degree} for each $v \in V(G)$:
\[\overline{\deg}_{\mathcal{H}}(v) \defeq \frac{1}{|L(v)|}\sum_{c \in L(v)}\deg_H(c).\]
Roughly speaking, the average color-degree is easier to concentrate than the maximum color-degree because the probability that the event described above occurs for \textit{all} $c$ in a list is relatively small. For a further discussion of why considering the average color-degree is useful, see the discussion by Pettie and Su in \cite[\S1]{PS15}.
Once we show $\deg_{\mathcal{H}'}(v)$ is concentrated, we remove any color $c \in L'(v)$ satisfying $\deg_{H'}(c) \geq 2\overline{\deg}_{\mathcal{H}'}(v)$.
In this way, we have an upper bound on $\Delta(H')$ as desired.
We remark that at this step we remove possibly up to half the vertices in $L'(v)$, which is why we require $C > 4$ in Theorem~\ref{theo:main_theo}.

A curious aspect of the analysis of this coloring procedure is that the local sparsity is only employed in the computation of $\E[\overline{\deg}_{H'}(v)]$.
In particular, as $H$ is $(k, K_{s, t})$-locally-sparse, the subgraph induced by $N_H(c)$ contains ``few'' edges.
As a result, the event $\set{c \in V(H')}$ and the random variable $\deg_{H'}(c)$ are ``nearly independent.''
This is somewhat surprising as for most applications of the \hyperref[algorithm: wcp]{Wasteful Coloring Procedure}, the concentration is considered to be the ``harder'' part of the analysis.

In the proof of Theorem~\ref{theo:DKPS}, Davies, Kang, Pirot, and Sereni obtained an upper bound on the number of edges in $G[N_G(v)]$ for a $(k, P_t)$-locally-sparse graph $G$ and a vertex $v \in V(G)$.
In fact, the bound follows from the following more general statement:

\begin{fact}\label{fact:edges}
    Let $G$ be a $(k, F)$-locally-sparse graph.
    For any $v \in V(G)$, we have
    \[e(G[N(v)]) \leq k + \text{\normalfont ex}(\deg_G(v), F),\]
    where $\text{\normalfont ex}(n, F)$ is the maximum number of edges in an $n$-vertex $F$-free graph.
\end{fact}

The proof of the above statement follows from the fact that one can make the subgraph $G[N(v)]$ $F$-free by removing at most $k$ edges.
Note that the bound provided by Fact~\ref{fact:edges} is only useful when $k \ll \deg(v)^2$.
This contributes to the necessary upper bound on $k$ in Theorems~\ref{theo:DKPS} and~\ref{theo:DKPSCorrespondence}.
We remark that Fact~\ref{fact:edges} in conjunction with Theorem~\ref{theo:KST} is sufficient to prove Theorem~\ref{theo:main_theo} when $C = 4+\eps$ and $k \leq d^{\eps/200}$.
The following result allows us to provide a sufficient bound on $e(G[N(v)])$ for larger values of $k$ when $C=8$.

\begin{prop}\label{prop:KST_multiple}
    Let $G$ be an $n$-vertex $(k, K_{s, t})$-sparse graph.
    For $k^\star = \frac{(s+t)^2\,n^{s+t - 1}}{2\,s^s\,t^t}$, we have
    \[|E(G)| \leq \left\{\begin{array}{cc}
        4\,n^{2-1/(st)} &  \text{ if } k \leq k^\star;\vspace{5pt} \\
        2\,s^{1/t}\,t^{1/s}\,k^{1/st}\,n^{2-1/s - 1/t} &  \text{ if } k > k^\star.
    \end{array}\right.\]
\end{prop}

We remark that we always apply the above proposition with $k \leq k^\star$, however, we include the case $k > k^\star$ for completeness.
Alon considered the case that $s,\, t = \Theta(1)$ in order to prove a supersaturation-type result.
In fact, for this range of $s$ and $t$, the contrapositive of \cite[Corollary 2.1]{alon2002testing} implies Proposition~\ref{prop:KST_multiple} for $k = \Theta(n^{s+t})$.
Our proof follows a similar strategy to that of \cite[Corollary 2.1]{alon2002testing} (see \S\ref{section: proof of KST multiple} for the details).

Once we are able to prove concentration, it remains to show that a ``good'' partial coloring $\phi$ exists.
To do so, we employ the symmetric version of the Lov\'asz Local Lemma.

\begin{theo}[{Lov\'asz Local Lemma; \cite[\S4]{MolloyReed}}]\label{LLL}
    Let $A_1$, $A_2$, \ldots, $A_n$ be events in a probability space. Suppose there exists $p \in [0, 1)$ such that for all $1 \leq i \leq n$ we have $\P[A_i] \leq p$. Further suppose that each $A_i$ is mutually independent from all but at most $\dlll$ other events $A_j$, $j\neq i$ for some $\dlll \in \N$. If $4p\dlll \leq 1$, then with positive probability none of the events $A_1$, \ldots, $A_n$ occur.
\end{theo}

In their proof of Theorem~\ref{theo:our_old_result}, the first two authors and Bernshteyn provided a framework for applying this version of the \hyperref[algorithm: wcp]{Wasteful Coloring Procedure}, i.e., the version for correspondence coloring in the color-degree setting.
Many of the arguments, particularly certain computational ones, can therefore be applied in a ``black box'' manner.
The main differences in the statements of Theorems~\ref{theo:our_old_result} and \ref{theo:main_theo} are the structural constraints on the cover graph $H$, and the sizes of the lists defined by $L$.
It turns out these assumptions appear in two key steps of the proof: first, the computation of $\E\left[\overline{\deg}_\mathcal{H'}(v)\right]$; second, the inductive application of the procedure, i.e., ensuring the list sizes are not too small when we reach the stage where we may apply Proposition~\ref{prop:final_blow}. 
The heart of our argument lies in the proofs of these two steps, which appear in \S\ref{sec: proof of expectation} and \S\ref{sec:main_proof}, respectively.

We conclude this subsection with a discussion on the bound $k \leq d^{(s+t)/10}$ in Theorem~\ref{theo:main_theo}.
When applying Proposition~\ref{prop:KST_multiple}, we require $k$ to be asymptotically smaller than $d^{s+t}$.
If not, it could be the case that $N_H(c)$ induces a very dense graph in $H$.
The events $\set{c' \in V(H')},\,\set{c \in V(H')}$ for $c' \in N_H(c)$ would then be strongly correlated.
As a result, the bound on $\E[\overline{\deg}_{H'}(v)]$ would not be sufficient to apply the nibble method with the desired list size $|L(v)|$.
We note that we require $k \ll d_i^{s+t}$, where $d_i$ is the relevant parameter during the $i$-th application of the method.
For $C > 4$, we are able to compute a lower bound  $d_i \geq d^{\gamma}$ for some appropriate $\gamma \defeq \gamma(C) > 0$.
As a result, it is enough to have $k \leq d^{r(s+t)}$ for some positive $r \defeq r(\gamma) < 1$.
We set $r = 1/10$ and $C = 8$ purely for computational ease.
We remark that we may decrease $C$ by decreasing $r$ to get a better bound on the chromatic number, however, we must have $C > 4$.

\section{The Wasteful Coloring Procedure}\label{sec:wcp}

This section will be split into three subsections.
In the first, we provide a description of our coloring procedure and state a key lemma at the heart of our proof of Theorem~\ref{theo:main_theo}.
In the second, we prove this result modulo a technical lemma, which we prove in the final subsection.

\subsection{Algorithm overview}\label{subsec: alg overview}

In this section, we formally describe our coloring procedure and introduce a key lemma at the heart of our proof of Theorem~\ref{theo:main_theo}.
Roughly speaking, the lemma below will allow us to, in \S\ref{sec:main_proof}, inductively color our graph until a point at which Proposition \ref{prop:final_blow} can be applied. 
As mentioned earlier, we will utilize the \hyperref[algorithm: wcp]{Wasteful Coloring Procedure} (described formally in Algorithm~\ref{algorithm: wcp}) in order to prove the lemma.
We first introduce some notation.

Recall from \S\ref{sec: proof overview} for a fixed vertex $v \in V(G)$, its average color-degree is
\[ \overline{\deg}_\mathcal{H}(v) = \frac{1}{|L(v)|}\sum_{c\in L(v)} \deg_H(c),\]
and the list of available colors for $v$ with respect to a partial coloring $\phi$ is as follows: 
\[ L_{\phi}(v) = \set{c \in L(v) \,:\, N_H(c)\cap \im(\phi) = \0},\]
where $\im(\phi)$ is the image of $\phi$, i.e., the colors assigned to colored vertices in $G$.
The main component of the proof of Theorem~\ref{theo:main_theo} is the following lemma, which shows that under certain conditions on $G$ and its correspondence cover, there exists a partial coloring such that the uncolored graph has desirable properties.

\begin{Lemma}\label{lemma:iteration} 
    There are $\tilde{d} \in \N$, $\tilde{\alpha} > 0$ such that the following holds. Suppose $\eta\in \R$, $d$, $\ell$, $s$, $t \in \N$, and $k \in \R$ satisfy:
    \begin{enumerate}[label=\ep{\normalfont L\arabic*}]

        \item\label{item:d_large_3.1} $d$ is sufficiently large: $d \geq \tilde{d}$,
        \item\label{item: k not too large} $k$ is not too large: $1/2 \leq k \leq d^{(s+t)/5}$,
        \item\label{item:ell} $\ell$ is bounded below and above in terms of $d$: $4\eta\, d < \ell < 100d$,
        \item\label{item:t} $s$ and $t$ are bounded in terms of $d$: $1 \leq t \leq s$ and $st \leq \dfrac{\tilde \alpha\log d}{\log \log d}$, and
        \item\label{item:eta} $\eta$ is bounded below and above in terms of $d$, $k$, $s$, and $t$: $\dfrac{1}{\log^5(d)} < \eta < \dfrac{1}{\log (dk^{-1/(s+t)})}.$ 
    \end{enumerate}
    Let $G$ be a graph with a correspondence cover $\mathcal{H} = (L, H)$ such that for some $\beta$ satisfying $d^{-1/(200st)} \leq \beta \leq 1/10$, 
    \begin{enumerate}[label=\ep{\normalfont L\arabic*},resume]
        \item \label{item: sparsity} $H$ is $(k,K_{s,t})$-locally-sparse,
        \item\label{item:Delta} $\Delta(H) \leq 2d$,
        \item\label{item:list_assumption} the list sizes are roughly between $\ell/2$ and $\ell$: $(1-\beta)\ell/2 \,\leq\, |L(v)| \leq (1+\beta)\ell$ for all $v \in V(G)$,
        \item\label{item:averaged} average color-degrees are smaller for vertices with smaller lists of colors: \[\overline{\deg}_\mathcal{H}(v) \,\leq\, \left(2 - (1 - \beta)\frac{\ell}{|L(v)|}\right)d \quad \text{for all } v\in V(G).\]
    \end{enumerate}
    Then there exists a proper partial $\mathcal{H}$-coloring $\phi$ of $G$ and 
    an assignment of subsets $L'(v) \subseteq L_\phi(v)$ to each $v \in V(G) \setminus \dom(\phi)$ with the following properties. Let
    \[
        G' \defeq G\left[V(G)\setminus \dom(\phi)\right] \qquad \text{and} \qquad H' \defeq H\left[\textstyle\bigcup_{v \in V(G')} L'(v)\right].
    \]
    Define the following quantities:
    \[
        \begin{aligned}[c]
            \keep &\defeq \left(1 - \frac{\eta}{\ell}\right)^{2d}, \\
        \uncolor &\defeq \left(1  - \frac{\eta}{\ell}\right)^{\keep\,\ell/2},
        \end{aligned}
        \qquad
        \begin{aligned}[c]
            \ell' &\defeq \keep\, \ell, \\
            d' &\defeq \keep\, \uncolor\, d,
        \end{aligned} \qquad \beta' \defeq (1+36\eta)\beta.
    \]
    Let $\mathcal{H}' \defeq (L', H')$, so $\mathcal{H}'$ is a correspondence cover of $G'$.
    Then for all $v \in V(G')$:
    \begin{enumerate}[label=\ep{\normalfont\roman*}]
        \item\label{item:I} $|L'(v)| \,\leq\, (1+\beta')\ell'$,
        
        \smallskip
        
        \item\label{item:II} $|L'(v)| \geq (1-\beta')\ell'/2$,
        
        \smallskip
        
        \item\label{item:III} $\Delta(H') \leq 2d'$,
        
        \smallskip
        
        \item\label{item:IV} $\overline{\deg}_{\mathcal{H}'}(v) \leq \left(2 - (1 - \beta')\frac{\ell'}{|L'(v)|}\right)d'.$

        \item\label{item: V} $H'$ is $(k, K_{s,t})$-locally-sparse.
    \end{enumerate}
\end{Lemma}

Note that condition \ref{item: V} holds automatically as $H' \subseteq H$. Also note that conditions \ref{item:I}--\ref{item:IV}  are similar to the conditions \ref{item: sparsity}--\ref{item:averaged}, except that the former uses $\beta', \eta', \ell'$. This will help us to apply Lemma~\ref{lemma:iteration} iteratively in \S\ref{sec:main_proof} to prove Theorem~\ref{theo:main_theo}. 
For the rest of this section, we define $\keep,\,\uncolor,\, d',\,\ell',\,\beta'$ as above.

Let us now describe the \hyperref[algorithm: wcp]{Wasteful Coloring Procedure} as laid out by the first two authors and Bernshteyn in \cite{anderson2022coloring} for correspondence colorings (the color-degree version for list coloring appeared in \cite{alon2020palette}).

\vspace{0.3cm}
\begin{breakablealgorithm}
\caption{Wasteful Coloring Procedure}\label{algorithm: wcp}
\begin{flushleft}
\textbf{Input}: A graph $G$ with a correspondence cover $\mathcal{H} = (L,H)$ and parameters $\eta \in [0,1]$ and $d$, $\ell> 0$. \\
\textbf{Output}: A proper partial $\mathcal{H}$-coloring $\phi$ and subsets $L'(v) \subseteq L_\phi(v)$ for all $v \in V(G)$.
\end{flushleft}
\begin{enumerate}[itemsep = .2cm, label = {\arabic*.}]
    \item Sample $A \subseteq V(H)$ as follows: for each $c \in V(H)$, include $c \in A$ independently with probability $\eta/\ell$.
    Call $A$ the set of \emphd{activated} colors, and let $A(v) \defeq L(v) \cap A$ for each $v \in V(G)$.

    \item Let $\{\eq(c) \,:\, c \in V(H)\}$ be a family of independent random variables with distribution
    \[\eq(c) \sim \Ber\left(\frac{\keep}{\left(1 - \eta/\ell\right)^{\deg_H(c)}}\right).\]
    (We discuss why this is well defined below). Call $\eq(c)$ the \emphd{equalizing coin flip} for $c$.
    
    \item Sample $K \subseteq V(H)$ as follows: for each $c \in V(H)$, include $c \in K$ if $\eq(c) = 1$ and $N_H(c) \cap A = \0$.
    Call $K$ the set of \emphd{kept} colors, and $V(H) \setminus K$ the \emphd{removed} colors. For each $v \in V(G)$, let $K(v) \defeq L(v) \cap K$.

    \item Construct $\phi : V(G) \pto V(H)$ as follows: if $A(v) \cap K(v) \neq \0$, set $\phi(v)$ to any color in $A(v) \cap K(v)$. Otherwise set $\phi(v) = \blank$.
    
    \item Call $v \in V(G)$  \emphdef{uncolored} if $\phi(v) = \blank$, and define 
    \[U \,\defeq\, \left\{c \in V(H)\,:\, \phi\left(L^{-1}(c)\right) = \blank\right\}.\]
    \ep{Recall that $L^{-1}(c)$ denotes the underlying vertex of $c$ in $G$.}

    \item\label{step: define L'} For each vertex $v \in V(G)$, let
    \[L'(v) \,\defeq\, \left\{c\in K(v)\,:\, |N_H(c) \cap K \cap U| \leq 2\,d'\right\}.\]
\end{enumerate}
\end{breakablealgorithm}
\vspace{0.3cm}

Note that if $G$ and $H$ satisfy assumption \ref{item:Delta} of Lemma \ref{lemma:iteration}, i.e. $\deg_H(c) \leq 2d$, then
\[ 0 \leq \frac{\keep}{(1 - \eta/ \ell)^{\deg_H(c)}} = \left( 1 - \frac{\eta}{\ell}\right)^{2d - \deg_H(c)} \leq 1\]
and hence the equalizing coin flips are well defined. 
Furthermore, if we assume that the assumptions of Lemma \ref{lemma:iteration} hold on the input graph of the \hyperref[algorithm: wcp]{Wasteful Coloring Procedure}, then $\keep$ is precisely the probability that a color $c \in V(H)$ is kept, and $\uncolor$ is roughly an upper bound on the probability that a vertex $v \in V(G)$ is uncolored.

\subsection{Proof of Lemma \ref{lemma:iteration}}\label{sec: iteration}
In this section, we prove Lemma \ref{lemma:iteration} under the assumption that a subsequently introduced lemma, Lemma \ref{expectationKeptUncolor}, is true. Indeed, Lemma \ref{expectationKeptUncolor} is the key lemma in our argument; it is the part of the proof that relies on the local sparsity condition; and the bulk of the work in proving Lemma \ref{lemma:iteration} comes from proving Lemma \ref{expectationKeptUncolor}, which we will do in \S\ref{sec: proof of expectation}. 

The rest of the proof of Lemma \ref{lemma:iteration} follows the strategy employed by the first two authors and Bernshteyn to prove \cite[Lemma 3.1]{anderson2022coloring}, which is similar to Lemma \ref{lemma:iteration}. 
Indeed, the assumptions of \cite[Lemma 3.1]{anderson2022coloring} and Lemma \ref{lemma:iteration} differ in only four ways: 
\begin{enumerate}
    \item the assumption that $H$ is $K_{1, s,t}$-free is replaced by the weaker assumption that $H$ is $(k, K_{s, t})$-locally-sparse,
    \item the assumption that $\eta < \frac{1}{\log d}$ is replaced by the assumption that $\eta < \frac{1}{\log\left(dk^{-1/(s+t)}\right)}$,
    \item the bounds $s \leq d^{1/10}$ and $t \leq \dfrac{\tilde \alpha\log d}{\log \log d}$ are replaced by $st \leq \dfrac{\tilde \alpha\log d}{\log \log d}$, and
    \item the assumption $d^{-1/(200t)} \leq \beta \leq 1/10$ is replaced by $d^{-1/(200st)} \leq \beta \leq 1/10$.
\end{enumerate}
However, several of the lemmas used in the proof of \cite[Lemma 3.1]{anderson2022coloring} do not rely on either of these assumptions, and thus the same proofs can be used for our lemmas here.
Furthermore, as our upper bound on $st$ matches that of $t$ in \cite[Lemma 3.1]{anderson2022coloring}, the proofs requiring the lower bound on $\beta$ follow by replacing $t$ with $st$ at the relevant steps.
This essentially allows us to use these proofs in a ``black box'' manner towards proving Lemma \ref{lemma:iteration}. We therefore follow the strategy of the proof of \cite[Lemma 3.1]{anderson2022coloring}, and indeed, this section follows similar notation to that of \cite[\S4]{anderson2022coloring}.

Assume parameters $\eta$, $d$, $\ell$, $s$, $t$, $k$, $\beta$, and a graph $G$ with correspondence cover $\mathcal{H} = (L, H)$ satisfy the assumptions of Lemma~\ref{lemma:iteration}.  We may assume that $\Delta(G) \leq 2(1+\beta) \ell d$ by removing the edges of $G$ whose corresponding matchings in $H$ are empty. Let the quantities $\keep$, $\uncolor$, $d'$, $\ell'$, and $\beta'$ be defined as in the statement of Lemma~\ref{lemma:iteration}. Suppose we have carried out the \hyperref[algorithm: wcp]{Wasteful Coloring Procedure} with these $G$ and $\mathcal{H}$. As in the statement of Lemma~\ref{lemma:iteration}, we let
    \[
        G' \defeq G\left[V(G)\setminus \dom(\phi)\right],  \qquad H' \defeq H\left[\textstyle\bigcup_{v \in V(G')} L'(v)\right], \qquad \text{and} \qquad \mathcal{H}' \defeq (L', H').
    \]
    For each $v \in V(G)$, we define the following quantities:
    \[
        \ell(v) \defeq |L(v)| \qquad \text{and} \qquad \overline{\deg}(v) \defeq \overline{\deg}_\mathcal{H}(v),
    \]
    as well as the following random variables:
    \[
        k(v) \defeq |K(v)|, \qquad  \ell'(v) \defeq |L'(v)|, \qquad \text{and} \qquad \overline{d}(v) \defeq \frac{1}{\ell'(v)}\sum_{c\in L'(v)}|N_H(c) \cap V(H')|.
    \]
    By definition, if $v \in V(G')$, then $\overline{d}(v) = \overline{\deg}_{\mathcal{H}'}(v)$. Our goal is to verify that statements \ref{item:I}--\ref{item:IV} in Lemma \ref{lemma:iteration} hold for every $v \in V(G')$ with positive probability. We follow an idea of Pettie and Su \cite{PS15} (which was also used in \cite{alon2020palette, anderson2022coloring}) by defining the following auxiliary quantities:
    \begin{align*}
        \lambda(v) &\defeq \frac{\ell(v)}{\ell}, & \lambda'(v)&\defeq \frac{\ell'(v)}{\ell'}, \\
        \delta(v) &\defeq \lambda(v)\,\overline{\deg}_{\mathcal{H}}(v) + (1-\lambda(v))\,2d, & \delta'(v) &\defeq \lambda'(v)\,\overline{d}(v) + (1-\lambda'(v))\,2d'.
    \end{align*}
    Note that, by \ref{item:list_assumption}, we have $(1-\beta)/2 \leq \lambda(v) \leq 1+ \beta$. When $\lambda(v) \leq 1$, we can think of $\delta(v)$ as what the average color-degree of $v$ would become if we added $\ell - \ell(v)$ colors of degree $2d$ to $L(v)$. \ep{We remark that in both \cite{PS15} and \cite{alon2020palette}, the value $\lambda(v)$ is artificially capped at $1$. However, as noted by the first two authors and Bernshteyn in \cite{anderson2022coloring}, it turns out that there is no harm in allowing $\lambda(v)$ to exceed $1$, which moreover makes the analysis simpler.} The upper bound on $\overline{\deg}(v)$ given by \ref{item:averaged} implies that
    \begin{align}\label{eqn:avg_delta}
        \delta(v) \,\leq\, (1+\beta)d.
    \end{align}
    As demonstrated in \cite[Lemma 4.1]{anderson2022coloring}, an upper bound on $\delta'(v)$ suffices to derive statements \ref{item:II}--\ref{item:IV} in Lemma \ref{lemma:iteration}:

\begin{Lemma}\label{lemma:deltaprime}
    If $\delta'(v) \leq (1 + \beta')d'$, then conditions \ref{item:II}--\ref{item:IV} of Lemma \ref{lemma:iteration} are satisfied.
\end{Lemma}
\begin{proof}
    The proof is the same as the proof of \cite[Lemma 4.1]{anderson2022coloring}.
\end{proof}

Therefore, to prove Lemma \ref{lemma:iteration}, it suffices to show that, with positive probability, the outcome of the \hyperref[algorithm: wcp]{Wasteful Coloring Procedure} satisfies $\delta'(v) \leq (1 + \beta')d'$ and  $\ell'(v) \leq (1+\beta')\ell'$ for all $v \in V(G)$. We shall now prove some intermediate results. Before we do so, consider the following inequality, which will be useful for proving certain bounds and follows for small enough $\tilde\alpha$:

\begin{align}\label{etabeta}
    \eta\beta \,\geq\, d^{-1/(200st)}/\log^5d \,\geq\, d^{-1/(100st)}.
\end{align}

In the next lemma, we show that $k(v)$ is concentrated around its expected value, which we then use to show condition \ref{item:I} in Lemma \ref{lemma:iteration} is satisfied with high probability.
The proof follows by a rather standard application of Talagrand's inequality (a powerful concentration tool).
We omit the proof here as it is identical to the one of \cite[Lemma 4.2]{anderson2022coloring}.

\begin{Lemma}\label{listConcentration}
    $\P[|k(v) -\keep\,\ell(v)| \geq \eta\beta\,\keep\,\ell(v)] \leq \exp\left(-d^{1/10}\right)$.
\end{Lemma}
\begin{proof}
The proof is the same as the proof of \cite[Lemma 4.2]{anderson2022coloring}.
\end{proof}

Since $\ell'(v)\leq k(v)$, we have the following with probability at least $1-\exp\left(-d^{1/10}\right)$:
\begin{align*}
    \ell'(v) &\leq (1+\eta\beta)\,\keep\,\ell(v) \\
    &\leq (1+\eta\beta)\,(1+\beta)\,\keep\,\ell \\
    &\leq (1+ \beta')\,\ell'.
\end{align*}
This implies that condition \ref{item:I} is met with probability at least $1 - \exp(-d^{1/10})$. 

In order to analyze the average color-degrees in the correspondence cover $\mathcal{H}'$, we define:
\begin{align*}
    \keptedges &\defeq \{cc'\in E(H)\,:\, c\in L(v),\ c'\in N_H(c), \text{ and } c,c'\in K\}, \\
    \uncoloredges &\defeq \{cc'\in E(H)\,:\, c\in L(v),\ c'\in N_H(c), \text{ and }
    c' \in U\}, \\
    \nd &\defeq \frac{|\keptedges \cap \uncoloredges|}{k(v)}.
\end{align*}
Note that $\nd$ is what the average color-degree of $v$ would be if instead of removing colors with too many neighbors on step~\hyperref[step: define L']{6} of the \hyperref[algorithm: wcp]{Wasteful Coloring Procedure}, we had just set $L'(v) = K(v)$.
The only place in our proof that relies on $H$ being $(k, K_{s,t})$-locally-sparse is the following lemma, which gives a bound on the expected value of $|\keptedges \cap \uncoloredges|$. 
As mentioned in \S\ref{sec: proof overview}, this lemma is one of the major differences between our proof and that of \cite{anderson2022coloring}.

\begin{Lemma}\label{expectationKeptUncolor}
$\E[|\keptedges\cap\uncoloredges|] \leq \keep^2\,\uncolor\,\ell(v)\,\overline{\deg}_{\mathcal{H}}(v)(1+6\eta\beta)$.
\end{Lemma}

\begin{proof}
See \S\ref{sec: proof of expectation}.
\end{proof}

The next lemma establishes that with high probability the quantity $|\keptedges\cap\uncoloredges|$ is not too large. The proof is identical to that of \cite[Lemma 4.4]{anderson2022coloring}, and thus we omit the technical details here. However, we describe the main idea of the proof below for the unfamiliar reader. 

The simplest approach to concentrate $|\keptedges\cap\uncoloredges|$ is to use a suitable version of Chernoff's inequality (see \cite[Claims A.6, A.7]{alon2020palette}, \cite[Lemmas 6, 7]{PS15}).
This approach fails for correspondence colorings as in this setting two colors from the same list can be connected by a path, and thus events defined on those colors are no longer independent. 
To deal with this, one can apply Talagrand's inequality, a powerful concentration tool for random variables satisfying certain Lipschitz-like constraints.
Unfortunately, the Lipschitz parameter of the random variable $|\keptedges\cap\uncoloredges|$ is too large to apply Talagrand's inequality directly.
Instead, the authors show that with sufficiently high probability $|\uncoloredges|$ is not much larger than its expected value, and $|\uncoloredges \setminus \keptedges|$ is not much smaller than its expected value.
The identity $|\keptedges\cap\uncoloredges| = |\uncoloredges| - |\uncoloredges \setminus \keptedges|$ then implies the desired concentration bound. 

While considering the random variable $|\uncoloredges|$, it turns out that the Lipschitz parameter is still too large to apply Talagrand's inequality.
However, the events that drive up the value of this parameter occur with very low probability.
Concentration can therefore be achieved by employing a variation of Talagrand's Inequality, which can handle such exceptional events, developed by Bruhn and Joos \cite[Theorem 3.1]{BRUHN2015277}.

For the quantity $|\uncoloredges \setminus \keptedges|$, we cannot apply this result of Bruhn and Joos directly, i.e., the Lipschitz parameter is still too large. 
To circumvent this problem, the authors employ a \textit{random partitioning technique} developed in their earlier work \cite{anderson2023colouring}. 
The idea is to find a partition $\set{X_1, \ldots, X_\tau}$ of the edge set $\cup_{c\in L(v)}E_H(c)$ such that $|(\uncoloredges \setminus \keptedges) \cap X_i|$ has a sufficiently small Lipschitz parameter for each $i$. 
It turns out a random partition satisfies this condition with high probability for an appropriate choice of $\tau$.
The proof now follows by applying the exceptional version of Talagrand's inequality to each part separately and then taking a union bound.

The result of the lemma is stated below, where $d_{\max} \defeq \max\{d^{7/8}, \Delta(H)\}$.

\begin{Lemma}\label{concentrationKeptUncolor}
$\P[|\keptedges\cap\uncoloredges| \geq \keep^2\,\uncolor\,\ell(v)(\overline{\deg}_{\mathcal{H}}(v) + 8\eta\beta\,d_{\max})] \leq d^{-100}$.
\end{Lemma}

\begin{proof}
The proof is the same as the proof of \cite[Lemma 4.4]{anderson2022coloring}.
\end{proof}

The next two lemmas are purely computational, using the results of Lemmas~\ref{listConcentration} and \ref{concentrationKeptUncolor}.
As the calculations are identical to those in \cite[Section 4]{anderson2022coloring}, we omit them here.
First, we show that $\nd$ is not too large with high probability.

\begin{Lemma}\label{degreeConcentration}
    $\P[\nd > \keep\,\uncolor\,\overline{\deg}_{\mathcal{H}}(v) + 15\eta\beta\,\keep\,\uncolor\,d_{\max}] \leq d^{-75}$.
\end{Lemma}
\begin{proof}
The proof is the same as the proof of \cite[Lemma 4.5]{anderson2022coloring}.
\end{proof}

We can now combine the results of this section to prove the desired bound on $\delta'(v)$. Again, the proof is purely computational, and so we refer to \cite[Lemma 4.6]{anderson2022coloring}.

\begin{Lemma}\label{deltaBound}$\P\left[\delta'(v) \leq (1+\beta')d'\right] \geq 1-d^{-50}$.
\end{Lemma}

\begin{proof}
The proof is the same as the proof of \cite[Lemma 4.6]{anderson2022coloring}.
\end{proof}

We are now ready to finish the proof of Lemma \ref{lemma:iteration} \ep{modulo Lemma~\ref{expectationKeptUncolor}}.

\begin{proof}[Proof of Lemma \ref{lemma:iteration}]
    This proof is the same as the proof of \cite[Lemma 3.1]{anderson2022coloring}, but is included here for completeness. We perform the \hyperref[algorithm: wcp]{Wasteful Coloring Procedure} on $G$ and $\mathcal{H}$ and define the following random events for each $v \in V(G)$: 
\begin{enumerate}
    \item $A_v \defeq \set{\ell'(v) \leq (1+\beta')}$,
    \item $B_v \defeq \set{\delta'(v) \geq (1+\beta')d'}$.
\end{enumerate}
We now use the \hyperref[LLL]{\LLL} \ep{Theorem \ref{LLL}}. By Lemmas \ref{listConcentration} and \ref{deltaBound}, we have:
\[\P[A_v] \leq \exp\left(-d^{1/10}\right) \leq d^{-50}, \qquad
\P[B_v] \leq d^{-50}.\]
Let $p \defeq d^{-50}$.
Note that the events $A_v$, $B_v$ are mutually independent from the events of the form $A_u$, $B_u$, where $u \in V(G)$ is at distance more than $4$ from $v$. Since $\Delta(G) \leq 2(1+\beta)\ell d$, there are at most $2(2(1+\beta)\ell d)^4 \leq d^{10}$ events corresponding to the vertices at distance at most $4$ from $v$. So we let $\dlll \defeq d^{10}$ and observe that $4p\dlll = 4 d^{-40} < 1$. By the \hyperref[LLL]{\LLL}, with positive probability none of the events $A_v$, $B_v$ occur. By Lemma~\ref{lemma:deltaprime}, this implies that, with positive probability, the output of the \hyperref[algorithm: wcp]{Wasteful Coloring Procedure} satisfies the conclusion of Lemma~\ref{lemma:iteration}.
\end{proof}

\subsection{Proof of Lemma \ref{expectationKeptUncolor}}\label{sec: proof of expectation}

This is the only part of the proof that requires $(k, K_{s, t})$-local-sparsity.
We will first provide an upper bound for $\E[\keptEdges]$.

\begin{Lemma}\label{lemma: expectation}
    $\E[\keptEdges] \leq \keep^2 \ell(v) \overline{\deg}(v)(1 + \eta \beta)$
\end{Lemma}

\begin{proof}
    Note that 
    \[\keptEdges = \sum_{c \in L(v)} \sum_{c' \in N_H(c)} \mathbbm{1}_{\{c, c' \in K\}}.\]
    If $c \in K$, then two events occur: no color in $N_H(c)$ is activated, and $c$ survives its equalizing coin flip. Since activations and equalizing coin flips occur independently, we have:
    \begin{align*}
        \mathbb{E}[\keptEdges] &= \sum_{c \in L(v)} \sum_{c' \in N_H(c)} \mathbb{P}[c, c' \in K]\\
        &= \sum_{c \in L(v)} \sum_{c' \in N_H(c)} \mathbb{P}[\eq(c) = 1]\, \mathbb{P}[\eq(c') = 1] \left(1 - \frac{\eta}{\ell} \right)^{|N_H(c) \cup N_H(c')|}\\
        &= \sum_{c \in L(v)} \sum_{c' \in N_H(c)} \mathbb{P}[\eq(c) = 1]\, \mathbb{P}[\eq(c') = 1] \left(1 - \frac{\eta}{\ell} \right)^{|N_H(c)|+|N_H(c')| - |N_H(c)\cap N_H(c')|}\\
        &= \keep^2 \sum_{c \in L(v)} \sum_{c' \in N_H(c)} \left(1 - \frac{\eta}{\ell} \right)^{-|N_H(c) \cap N_H(c')|}.
    \end{align*}
    For $c \in H$, we will consider two cases based on the value of $\deg_H(c)$ (we will abuse notation and say $c \in \text{C1}$ and $c \in \text{C2}$ when $c$ satisfies the conditions to be in Case 1 and Case 2, respectively.)

\begin{enumerate}[leftmargin = \leftmargin + 1\parindent, wide, label=(Case \arabic*)] 
    \item\label{Case1}

    $\deg_H(c) < d^{1 - 1/(5st)}$. Then 
    \[|N_H(c) \cap N_H(c')| \leq |N_H(c)| = \deg_H(c) \leq d^{1 - 1/(5st)}.\]
    
    Thus, using the inequality $(1-x)^n \geq 1-nx$ for $n\geq 1$ and $x \leq 1$, and the lower bound $\ell \geq 4\eta d$ from assumption \ref{item:ell}, we have:
    \begin{align}
        \sum_{c' \in N_H(c)} \left(1 - \frac{\eta}{\ell} \right)^{-|N_H(c) \cap N_H(c')|} \nonumber  
        &\leq \sum_{c' \in N_H(c)} \left(1 - \frac{\eta}{\ell} \right)^{- d^{1 - 1/(5st)}} \quad \text{(as $1-\eta/\ell < 1$)} \nonumber\\
        &\leq \deg_H(c) \left(1 - d^{1 - 1/(5st)} \frac{\eta}{\ell} \right)^{-1} \nonumber \\
        &\leq \deg_H(c) \left( 1 - \frac{d^{- 1/(5st)}}{4} \right)^{-1} \nonumber\\
        &\leq \deg_H(c) \left( 1 + d^{- 1/(5st)} \right), \label{eq: Case 1}
    \end{align}
    where the last inequality holds since $d^{-1/(5st)} \leq 1$.

    \item\label{Case2}
    
    $\deg_H(c) \geq d^{1 - 1/(5st)}$. We note that this is the only place where we use $(k, K_{s,t})$-local-sparsity. Define the following sets:
    \begin{align*}
        \Bad &\defeq \left\{c' \in N_H(c) : |N_{H}(c') \cap N_H(c)|\geq \deg_H(c)^{1-1/(5st)}\right\},\\
        \Good &\defeq N_H(c) \setminus \Bad.
    \end{align*} 
    
    Then:
    \begin{align}
        &\hspace{1.2em}\sum_{c' \in N_H(c)} \left(1 - \frac{\eta}{\ell} \right)^{-|N_H(c) \cap N_H(c')|} \nonumber \\
        &= \sum_{c' \in \Bad} \left(1 - \frac{\eta}{\ell} \right)^{-|N_H(c) \cap N_H(c')|} + \sum_{c' \in  \Good} \left(1 - \frac{\eta}{\ell} \right)^{-|N_H(c) \cap N_H(c')|} \nonumber \\
        & \leq \sum_{c' \in \Bad} \left(1 - \frac{\eta}{\ell} \right)^{-\deg_H(c)} + \sum_{c' \in  \Good} \left(1 - \frac{\eta}{\ell} \right)^{- \deg_H(c)^{1 - 1/5st}} \nonumber \\ 
        & \leq     |\Bad| \left(1 - \frac{\eta}{\ell} \right)^{-2d} + (\deg_H(c) - |\Bad|) \left(1 - \frac{\eta}{\ell} \right)^{- \deg_H(c)^{1 - 1/5st}}. \label{eq:final inequality for bound with bad}
    \end{align}
    
    Since $\deg_H(c)^{1 - 1/5st} \leq 2d$, the last expression is increasing in terms of $|\Bad|$.
    The following claim provides an upper bound on the size of $\Bad$.
    
    \begin{claim}\label{claim:bad}
        $|\Bad| \leq \deg_H(c)^{1- 1/(5st)}$. 
    \end{claim}
    \begin{claimproof}
        By the Handshaking Lemma,
        \[ |\Bad| \leq \frac{2\,|E(H[N_H(c)])|}{\deg_H(c)^{1 - 1/(5st)}}.\]
        Since $H$ is $(k, K_{s,t})$-locally-sparse, $H[N_H(c)]$ is $(k, K_{s,t})$-sparse,
        and so by Proposition~\ref{prop:KST_multiple} it follows that if for
        $k^\star \defeq \dfrac{(s+t)^2\,\deg_H(c)^{s+t - 1}}{2\,s^s\,t^t}$  we have $k \leq k^\star$, then
        \[ |E(H[N_H(c)])| \leq 4\,\deg_H(c)^{2-1/(st)} .\]
        Let us first show $k \leq k^\star$.
        As we are in \ref{Case2}, we have $\deg_H(c) \geq d^{1 - 1/(5st)}$.
        It follows that
        \[k^\star \geq \dfrac{(s+t)^2\,d^{(1 - 1/(5st))(s+t - 1)}}{2\,s^s\,t^t}  = \dfrac{(s+t)^2\,d^{s+t - 1 - 1/(5t) - 1/(5s) + 1/(5st)}}{2\,s^s\,t^t}  \]
        Note that 
        \[\frac{s^s t^t}{(s+t)^2} < (st)^{st} \leq \left( \frac{\tilde{\alpha} \log d }{\log \log d}\right)^{\frac{\tilde{\alpha} \log d }{\log \log d}} < \log d^{\frac{\tilde{\alpha} \log d }{\log \log d}} = d^{\tilde{\alpha}},\]
        for $\tilde \alpha$ sufficiently small, where the second inequality holds by assumption \ref{item:t}.
        Thus, 
        \[\frac{(s+t)^2}{ s^s t^t} > d^{-\tilde{\alpha}}.\]
        Additionally, we claim that
        \[s+t - 1 - 1/(5t) - 1/(5s) + 1/(5st) \geq (s+t)/5 + 2/5, \quad  \text{for all} \quad s \geq t \geq 1.\] 
        One can manually verify the above for $s = t = 1$ and $s = 2, t = 1$. 
        For $s \geq t \geq 2$,  we have
        \[ 4(s+t)/5 + 1/(5st) > 3 > 1/(5t) + 1/(5s) + 1 + 2/5,\]
        which implies the desired inequality.
        
        Thus, for $d$ large enough and $\tilde{\alpha}$ sufficiently small, we have
        $k^\star > d^{(s+t)/5} \geq k$ by assumption \ref{item: k not too large}.
        Hence, by Proposition~\ref{prop:KST_multiple} we have
        \begin{align*}
            |\Bad| \leq \frac{2\,|E(H[N_H(c)])|}{\deg_H(c)^{1 - 1/(5st)}}
            &\leq \frac{2\cdot 4\,\deg_H(c)^{2-1/(st)}}{\deg_H(c)^{1 - 1/(5st)} } \\
            &= 8 \deg_H(c)^{1 - 4/(5st)} \\
            &\leq \deg_H(c)^{1 - 1/(5st)},
        \end{align*}
        as desired.
    \end{claimproof}

    With Claim~\ref{claim:bad} in hand, and as $\left(1 - \frac{\eta}{\ell} \right)^{-2d} = \frac{1}{\keep} \leq 2$ due to assumption \ref{item:ell}, we have the following as a result of \eqref{eq:final inequality for bound with bad}:
    \begin{align}
            &\quad \sum_{c' \in N_H(c)} \left(1 - \frac{\eta}{\ell} \right)^{-|N_H(c) \cap N_H(c')|} \nonumber\\
            & \leq  2 \deg_H(c)^{1- 1/(5st)} +  (\deg_H(c) -  \deg_H(c)^{1- 1/(5st)} ) \left(1 - \frac{\eta}{\ell} \right)^{- \deg_H(c)^{1 - 1/5st}} \nonumber \\
            & \leq 2 \deg_H(c)^{1- 1/(5st)} +  \frac{(\deg_H(c) -  \deg_H(c)^{1- 1/(5st)} )}{ \left(1 - \deg_H(c)^{1 - 1/5st} \frac{\eta}{\ell} \right) } \nonumber  \\
            & = \deg_H(c)\left( 2 \deg_H(c)^{- 1/(5st)} +  \frac{1 -  \deg_H(c)^{- 1/(5st)} }{ \left(1 - \frac{\eta}{\ell} \deg_H(c)^{1 - 1/5st}  \right) }  \right) \nonumber \\
            & \leq  \deg_H(c)\left( 2 \deg_H(c)^{- 1/5st} +  \frac{1 -  \deg_H(c)^{- 1/(5st)} }{ \left(1 - \frac{\eta}{\ell} 2d \deg_H(c)^{- 1/5st}  \right) }  \right) \qquad (\text{since } \deg_H(c) \leq 2d) \nonumber\\
            & \leq  \deg_H(c)\left( 2 \deg_H(c)^{- 1/5st} +  \frac{1 -  \deg_H(c)^{- 1/5st} }{ \left(1 - \frac{1}{2} \deg_H(c)^{- 1/5st}  \right) }  \right) \qquad (\text{since } 2d \frac{\eta}{\ell} \leq \frac{1}{2}) \nonumber \\
            & \leq \deg_H(c)\left( \deg_H(c)^{- 1/10st} + 1  \right) \label{eq: Case 2},
    \end{align}
    completing this case.
\end{enumerate}

By \eqref{etabeta}, we have $\eta\beta \,\geq d^{-1/(100st)} \geq d^{-1/(10st)}$.
Therefore, putting together the bounds \eqref{eq: Case 1} from \ref{Case1} and \eqref{eq: Case 2} from \ref{Case2}, we have the following:
 \begin{align*}
    \mathbb{E}[\keptEdges]  &  = \keep^2 \sum_{c \in L(v)} \sum_{c' \in N_H(c)} \left(1 - \frac{\eta}{\ell} \right)^{-|N_H(c) \cap N_H(c')|} \\
    & \leq \keep^2  \left( \sum_{c \in \text{\hyperlink{Case1}{C1}}} \deg_H(c) \left( 1 + d^{- 1/(5st)} \right)     + \sum_{c \in \text{\hyperlink{Case2}{C2}}} \deg_H(c)\left( \deg_H(c)^{- 1/10st} + 1  \right)  \right)  \\
    & \leq \keep^2  \left( \sum_{c \in \text{\hyperlink{Case1}{C1}}} \deg_H(c) \left(  \eta \beta +1 \right)     + \sum_{c \in \text{\hyperlink{Case2}{C2}}} \deg_H(c)\left( \eta \beta + 1  \right)  \right) \\
    & = \keep^2 \sum_{c \in L(v)} \deg_H(c) \left(  \eta \beta +1 \right)  \\
    &= \keep^2 \,\ell(v)\, \overline{\deg}_{\mathcal{H}}(v)\left( \eta \beta + 1  \right),
\end{align*}

completing the proof.
\end{proof}

The following lemma bounds $\E[|\keptedges\cap\uncoloredges|]$ in terms of $\E[|\keptedges|]$.
The proof involves a special case of the FKG inequality, dating back to Harris and Kleitman.
As the argument is identical to that of \cite[Lemma 5.5]{anderson2022coloring}, we omit the details here.

\begin{Lemma}\label{keptUncolorExpectation}
$\E[|\keptedges\cap\uncoloredges|] \leq \uncolor\,(1+4\eta\beta)\E[|\keptedges|]$.
\end{Lemma}
\begin{proof}
    The proof is the same as that in \cite[Lemma 5.2]{anderson2022coloring}. 
\end{proof}

Lemmas~\ref{lemma: expectation} and \ref{keptUncolorExpectation} together imply that
\begin{align*}
    \E[|\keptedges\cap\uncoloredges|] &\leq \keep^2\,\uncolor\,\ell(v)\,\overline{\deg}_{\mathcal{H}}(v)(1+4\eta\beta)(1+\eta\beta) \\
    &\leq \keep^2\,\uncolor\,\ell(v)\,\overline{\deg}_{\mathcal{H}}(v)(1+6\eta\beta),
\end{align*}
which completes the proof of Lemma \ref{expectationKeptUncolor}.

\section{Proof of Theorem~\ref{theo:main_theo}}\label{sec:main_proof}

To prove Theorem \ref{theo:main_theo}, we start by defining several parameters: 
For a given $\epsilon >0$, $d, k \in \R, \,s, t \in \N\setminus\{0\}$, and $i \in \N$, we define the following:
\begin{align*}
    C &\defeq \begin{cases}
        4+\epsilon & \text{ when }k \leq d^{\epsilon(s+t)/200},\\
        8 & \text{ else}.  
    \end{cases}& &\\
    \mu &\defeq \frac{C-\epsilon}{2}\log\left(1+\frac{\epsilon}{8C}\right)  &\eta &\defeq \mu/\log \left(dk^{-1/(s+t)}\right)\\
    \ell_0 &\defeq C\,d/\log\left(dk^{-1/(s+t)}\right) & d_0 &\defeq d \\
    k_i &\defeq \left(1 - \frac{\eta}{\ell_i}\right)^{2d_i}
    & u_i &\defeq \left(1 - \frac{\eta}{\ell_i}\right)^{k_i\,\ell_i/2}
    \\
    \ell_{i+1} &\defeq k_i\, \ell_i & d_{i+1}&\defeq k_i\, u_i\, d_i \\
    \beta_0 &\defeq d_0^{-1/(200st)} & \beta_{i+1} &\defeq \max\left\{(1+36\eta)\beta_i,\, d_{i+1}^{-1/(200st)}\right\}.
\end{align*}
We will also assume $\eps$ is sufficiently small such that $\mu < 1$.
The reader may be familiar with similar arguments parameterized by $\keep_i$ and $\uncolor_i$.
For brevity, we use $k_i$ and $u_i$ instead.

\begin{Lemma}\label{lemma: iteration helper 1} For all $\epsilon > 0$ sufficiently small, there exists $d^*$ such that whenever $d \geq d^*$ and 
\[\frac{1}{2} < k \leq d^{(s+t)/10}, 
    \]
the following hold: \begin{enumerate}[label = {(I\arabic*)}]
    \item\label{item: ratio decreaes by uncolor} For $j \in \N$, $d_{j+1}/\ell_{j+1} \,\leq\, d_j/\ell_j \,\leq\, d_0/\ell_0 \,=\, \frac{1}{C}\log\left(dk^{-1/(s+t)}\right)$. 
    \item\label{item: lower bound on ell by power of d} For $j \in \N$, $\ell_j \,\geq\, d\left(dk^{-1/(s+t)}\right)^{-4/(C-7\epsilon/8)}$.
    \item\label{item: existence of i star} There exists a minimum integer $i^* \,\leq\, \left\lceil\frac{16}{\mu}\log \left(dk^{-1/(s+t)}\right)\log\log \left(dk^{-1/(s+t)}\right)\right\rceil$ such that $d_{i^\star}\,\leq\,\ell_{i^\star}/100$.
\end{enumerate}
\end{Lemma}

\begin{proof} 
    We will not explicitly compute the value $d^*$, but instead simply state that we let $d$ be sufficiently large when needed.
    
    By definition, $d_{j+1}/\ell_{j+1}$ is smaller than $d_{j}/\ell_{j}$ by a factor of $u_j$ and so \ref{item: ratio decreaes by uncolor} holds.
    Note the following:
    \begin{align}\label{eq:bounds on dk}
    0 < 4/(C-7\epsilon/8) < 1, \quad \text{and} \quad d^{9/10} \,\leq\, dk^{-1/(s+t)} \,\leq\, 2d.
    \end{align}
    We now use strong induction on $j \in \N$ to prove \ref{item: lower bound on ell by power of d}. 
    As $d \to \infty$, it follows $dk^{-1/(s+t)} \to \infty$. Thus for $d$ large enough, we have
    \[\frac{\log\left(dk^{-1/(s+t)}\right)}C \,\leq\, \left(dk^{-1/(s+t)}\right)^{4/(C-7\epsilon/8)},\]
    from where it follows that
    \[\ell_0 \,=\, \frac{C\,d}{\log\left(dk^{-1/(s+t)}\right)} \,\geq\, d\left(dk^{-1/(s+t)}\right)^{-4/(C-7\epsilon/8)}.\]
    This completes the proof of the base case. Now assume for some $j' \in \N$, \ref{item: lower bound on ell by power of d} holds
    for all $0 \leq j \leq j'$. We will show $\ell_{j'+1} \geq d\left(dk^{-1/(s+t)}\right)^{-4/(C-7\epsilon/8)}$.
    To do so, we will compute a sufficient lower bound on $k_j$ for each $0 \leq j \leq j'$, which will allow us to prove the desired bound on $\ell_{j'+1}$.
    We start by showing $k_{j}$ is bounded from below by a constant (independent of $j$); we then use this to compute an upper bound on $u_j$, which we then use to prove a stronger lower bound on $k_{j}$.
    
    Let $0 \leq j \leq j'$.
    First note
    \begin{align}
        \frac{\eta}{\ell_j} &= \frac{\mu}{\log\left(dk^{-1/(s+t)}\right)}\frac{1}{\ell_j} \nonumber \\
        &\leq \frac{\mu}{\log\left(dk^{-1/(s+t)}\right)}\frac{\left(dk^{-1/(s+t)}\right)^{4/(C-7\epsilon/8)}}{d} \nonumber \\
        &\leq \frac{10}{9\log(d)}\frac{(2d)^{4/(4+\epsilon/8)}}{d} \nonumber \\
        &\leq \frac{1}{d^{\eps/50}}, \label{eq: eta ell bound}
    \end{align}
    where the inequalities follow by \eqref{eq:bounds on dk}, the bound on $\ell_j$ from the induction hypothesis, and for $d$ large enough. 
    Now recall for any $\epsilon > 0$, there exists $x_0 > 0$ such that $1-x\geq \exp\left(-\frac{x}{1-\epsilon}\right)$ whenever $0 < x < x_0$. Thus for $d$ sufficiently large, by \eqref{eq: eta ell bound} it follows $\eta/\ell_j$ is sufficiently small for the following to hold:
    \[k_j = \left(1 - \frac{\eta}{\ell_j}\right)^{2d_j} \geq\, \exp\left(- \frac{\eta 2d_j}{\ell_j (1- \epsilon/C)}\right) \geq\, \exp\left(-\frac{2\eta d_0}{\ell_0(1-\epsilon/C)}\right),\]
    where the second inequality follows from \ref{item: ratio decreaes by uncolor}.
    
    We now have:
        \[\exp\left(-\frac{2\eta d_0}{\ell_0(1-\epsilon/C)}\right) \,=\, \exp\left(-\frac{2\mu}{(1-\epsilon/C)C}\right) \,=\, \exp\left(-\frac{2\mu}{C - \epsilon}\right).\]
    Thus \[k_j \,\geq\, \exp\left(-\frac{2\mu}{C-\epsilon}\right).\]
    Now we use this lower bound on $k_j$ to derive an upper bound on $u_j$ as follows:
    
    \begin{align*}
    u_j = \left(1  - \frac{\eta}{\ell_j}\right)^{k_j\ell_j/2} &\leq\, \exp\left(-\frac{\eta k_j \ell_j}{2\ell_j}\right) \,=\, \exp\left(-\frac{\eta k_j}{2}\right) \,\leq\, \exp\left(-\frac{\eta}{2}\exp\left(-\frac{2\mu}{C-\epsilon}\right)\right)\\
    & \leq\, 1 - \left(1-\epsilon/(4C)\right)\frac{\eta}{2}\exp\left(-\frac{2\mu}{C-\epsilon}\right),
    \end{align*}
    where the last inequality follows by the fact that for any $\epsilon > 0$, there exists $x_0 > 0$ such that $1-x\geq \exp\left(-\frac{x}{1-\epsilon}\right)$ whenever $0 < x < x_0$ (we assume $d$ is large enough so that we may apply this inequality).
    
    As the expression above is independent of $j$, we may use it to compute an upper bound on $\frac{d_j}{\ell_j}$ as follows:
    
    \begin{equation}\label{eq: upper bound on dj over lj}
        \frac{d_j}{\ell_j} \,=\, \frac{d_0}{\ell_0} \prod_{m = 0}^{j-1} u_m \,\leq\, \frac{d_0}{\ell_0}\left(1 - \left(1-\epsilon/(4C)\right)\frac{\eta}{2}\exp\left(-\frac{2\mu}{C - \epsilon}\right)\right)^{j}.
    \end{equation}
    We use this to improve our lower bound on $k_j$.
    Assuming $d$ is large enough, by \eqref{eq: eta ell bound} we have $\eta/\ell_j$ is small enough such that:
    \begin{align*}
        k_j &= \left(1 - \frac{\eta}{\ell_j}\right)^{2d_j} \geq\, \exp\left(- \frac{\eta}{\ell_j}\frac{2d_j}{(1-\epsilon/(2C))}\right)\\
        &\geq \exp\left(-\frac{2\eta}{(1- \epsilon/(2C))}
        \frac{d_0}{\ell_0}\left(1 - \left(1-\epsilon/(4C)\right)\frac{\eta}{2}\exp\left(-\frac{2\mu}{C-\epsilon}\right)\right)^{j}\right) \\
        & = \exp\left(-\frac{2\mu}{(1- \epsilon/(2C))C}
        \left(1 - \left(1-\epsilon/(4C)\right)\frac{\eta}{2}\exp\left(-\frac{2\mu}{C-\epsilon}\right)\right)^{j}\right)\\
        & = \exp\left(-\frac{2\mu}{C - \epsilon/2}
        \left(1 - \left(1-\epsilon/(4C)\right)\frac{\eta}{2}\exp\left(-\frac{2\mu}{C-\epsilon}\right)\right)^{j}\right),
    \end{align*}
    This bound holds for all $0 \leq j \leq j'$.
    We may now lower bound $\ell_{j'+1}$ as follows:
    \begin{align*}
        \ell_{j'+1} &= \ell_0 \prod_{m = 0}^{j'} k_m \,\geq\, \ell_0 \prod_{m=0}^{j'} \exp\left(-\frac{2\mu}{C-\epsilon/2}
        \left(1 - \left(1-\epsilon/(4C)\right)\frac{\eta}{2}\exp\left(-\frac{2\mu}{C-\epsilon}\right)\right)^{m}\right)\\
        & = \ell_0 \exp\left(-\frac{2\mu}{C-\epsilon/2}\sum_{m=0}^{j'}
        \left(1 - \left(1-\epsilon/(4C)\right)\frac{\eta}{2}\exp\left(-\frac{2\mu}{C-\epsilon}\right)\right)^{m}\right)\\
        & \geq \ell_0 \exp\left(-\frac{2\mu}{(C-\epsilon/2)}
        \frac{1}{\left(1-\left(1 - \left(1-\epsilon/(4C)\right)\frac{\eta}{2}\exp\left(-\frac{2\mu}{C-\epsilon}\right)\right)\right)}\right),
    \end{align*}
    where the last inequality follows by bounding the summation as a partial sum of a geometric series (we may assume that $d$ is large enough implying $\eta$ is small enough so this is indeed true). From this it follows:
    \begin{align*}
        \ell_{j'+1} & \geq \ell_0 \exp\left(-\frac{2\mu}{(C-\epsilon/2)}
        \frac{2}{(1-\epsilon/(4C))\eta\exp\left(-\frac{2\mu}{C-\epsilon} \right)}\right)\\
        &= \ell_0 \exp\left(\frac{-4\log\left(dk^{-1/(s+t)}\right)\exp\left(\frac{2\mu}{C-\epsilon}\right)}{(C-\epsilon/2)(1-\epsilon/(4C))}\right) \\
        & \geq \ell_0 \exp\left(\frac{-4\log\left(dk^{-1/(s+t)}\right)\exp\left(\frac{2\mu}{C-\epsilon}\right)}{C-3\epsilon/4}\right) \qquad \text{(for $\epsilon$ sufficiently small)}\\
        &= \ell_0(dk^{-1/(s+t)})^\xi \qquad \left(\text{where }\xi \defeq \frac{-4\exp\left(\frac{2\mu}{C-\epsilon}\right)}{C-3\epsilon/4} = \frac{-4(1+\epsilon/(8C))}{C-3\epsilon/4}\right)\\
        & =\frac{C\,d}{\log\left(dk^{-1/(s+t)}\right)}\left(dk^{-1/(s+t)}\right)^\xi\\
        & \geq d\left(dk^{-1/(s+t)}\right)^{-4/(C-7\epsilon/8)},
    \end{align*}
    where the last inequality follows for $d$ large enough as $\xi > -4/(C-7\epsilon/8)$.
    This completes the induction and finishes the proof of \ref{item: lower bound on ell by power of d}.

    For \ref{item: existence of i star}, note that by \eqref{eq: upper bound on dj over lj}, it follows for all $i \in \N$,
    \begin{align*}
        \frac{d_i}{\ell_i} &\leq \frac{d_0}{\ell_0}\left(1 - \left(1-\frac{\epsilon}{4C}\right)\frac{\eta}{2}\exp\left(-\frac{2\mu}{C-\epsilon}\right)\right)^{i}\\
        & \leq \frac{d_0}{\ell_0} \exp\left(-i\left(1-\frac{\epsilon}{4C}\right)\frac{\eta}{2}\exp\left(-\frac{2\mu}{C-\epsilon}\right)\right)\\
        & = \frac{\log\left(dk^{-1/(s+t)}\right)}{C}\exp\left(-i\left(1-\frac{\epsilon}{4C}\right)\frac{\eta}{2}\frac{1}{1+\epsilon/(8C)}\right)
        \\
        & \leq \frac{\log\left(dk^{-1/(s+t)}\right)}{C}\exp\left(-\frac{i\,\mu}{8\log \left(dk^{-1/(s+t)}\right)}\right),
    \end{align*} 
    for $\eps$ small enough.
    Now, for 
    \[j^* \defeq \left\lceil\frac{16}{\mu}\log \left(dk^{-1/(s+t)}\right)\log\log \left(dk^{-1/(s+t)}\right)\right\rceil,\]we have:
    \begin{align*}
        \frac{d_{j^*}}{\ell_{j^*}} & \leq \frac{\log\left(dk^{-1/(s+t)}\right)}{C}\exp\left(-\frac{j^*\,\mu}{8\log \left(dk^{-1/(s+t)}\right)}\right)\\
        & \leq \frac{\log\left(dk^{-1/(s+t)}\right)}{C}\exp\left(-2\log\log\left(dk^{-1/(s+t)}\right)\right) \\
        & \leq \frac{1}{C\,\log\left(dk^{-1/(s+t)}\right)} \\
        & \leq \frac{1}{100}
    \end{align*}
    for $d$ large enough by \eqref{eq:bounds on dk}. Thus a minimum integer $i^* \leq j^*$ such that $d_{i^*}/\ell_{i^*} \leq 1/100$ exists. 
\end{proof}

\begin{Lemma}\label{lemma: iteration helper 2}
    There exists $\alpha > 0$ such that the following holds: Let $\epsilon >0$ be sufficiently small. Then there exists $d^\symb$ such that whenever 
    \[d \geq d^\symb, \quad 1 \leq t \leq s, \quad st \,\leq\, \frac{\alpha\,\epsilon\,\log d}{\log \log d}, \quad \text{and} \quad 1/2 < k \,\leq\, d^{(s+t)/10},
        \]
        the following hold for all $i \in \N$ with $0 \leq i < i^*$, where $i^*$ is defined (and guaranteed to exist) by Lemma~\ref{lemma: iteration helper 1}~item~\ref{item: existence of i star}.
    \begin{enumerate}[label=\ep{\normalfont R\arabic*}, leftmargin = \leftmargin + \parindent]
        \item\label{item: di not too large} $d_i$ is sufficiently large: $d_i \geq \tilde{d}$, where $\tilde{d}$ is from Lemma \ref{lemma:iteration},
        \item\label{item: k not too large thm proof} $k$ is not too large: $k \leq d_i^{(s+t)/5}$.
        \item\label{item: ell bounded by d} $\ell_i$ is bounded below and above in terms of $d_i$: $4\eta\,d_i < \ell_i < 100d_i$,
        \item\label{item: s and t bounded} $s$ and $t$ are bounded in terms of $d_i$: $1 \leq t \leq s$ and $st \leq \dfrac{\tilde{\alpha}\log d_i}{\log\log d_i}$,
        \item\label{item: eta small} $\eta$ is sufficiently small: $\dfrac{1}{\log^5d_i} < \eta < \dfrac{1}{\log (d_ik^{-1/(s+t)})}$,
        \item\label{item: beta small} $\beta_i$ is small: $\beta_i \leq 1/10$. \ep{Note that the bound $\beta_i \geq d_i^{-1/(200st)}$ holds by definition.}
    \end{enumerate}
\end{Lemma}

\begin{proof} 
    Once again, we will not explicitly compute $d^\symb$, but instead simply take $d$ to be sufficiently large when needed.
    For all $i < i^*$, we have $d_i \geq \ell_i/100$, and by \ref{item: lower bound on ell by power of d} of Lemma \ref{lemma: iteration helper 1}, it follows for $d$ large enough that $\ell_i \geq d\left(dk^{-1/(s+t)}\right)^{-4/(C-7\epsilon/8)}$. Note the following as a result of \eqref{eq:bounds on dk}:
    \begin{equation}\label{eq: di bounded by d to epsilon over 40}
        d_i \geq \frac{1}{100}d\left(dk^{-1/(s+t)}\right)^{-4/(C-7\epsilon/8)} \geq d^{\epsilon/40}. \stepcounter{equation}\tag{\theequation}
    \end{equation}
    Using the above, note for $d^\symb$ large enough, namely $d^\symb \geq {\tilde{d}}^{40/\epsilon}$, it follows
    \[d_i \geq d^{\epsilon/40} \geq \tilde{d},\]
    showing \ref{item: di not too large}.

    From \eqref{eq: di bounded by d to epsilon over 40} it also follows that when $k \leq d^{\epsilon(s+t)/200}$, we have:
    \[k \,\leq\, d^{\epsilon(s+t)/200} \,\leq\, d_i^{40(s+t)/200} = d_i^{(s+t)/5}.\]
    Otherwise, when $k > d^{\epsilon(s+t)/200}$, we have $C = 8$. Using the first inequality in \eqref{eq: di bounded by d to epsilon over 40}, it follows that:
    \begin{align*}
        d_i^{(s+t)/5} &\geq \left(\frac{1}{100}d\left(dk^{-1/(s+t)}\right)^{-4/(C-7\epsilon/8)}\right)^{(s+t)/5}\\
        & = \left(\frac{1}{100} d^{\left(1-\frac{4}{8-7\epsilon/8}\right)}k^{\left(\frac{4}{(s+t)(8-7\epsilon/8)} \right)}  \right)^{(s+t)/5}\\
        & \geq \left(d^{\left(1-\frac{4}{8-7\epsilon/4}\right)}k^{\left(\frac{4}{(s+t)(8-7\epsilon/8)} \right)}  \right)^{(s+t)/5}\\
        & \geq  \left( k^{\frac{10}{s+t}\left(1-\frac{4}{8-7\epsilon/4}\right)+\frac{4}{(s+t)(8-7\epsilon/8)}}  \right)^{(s+t)/5} \qquad \left(\text{as } k \leq d^{(s+t)/10}\right) \\
        & = k^{2\,\left(1-\frac{4}{8-7\epsilon/4}\right)+\frac{4}{5(8-7\epsilon/8)}}.
    \end{align*}
    Let us consider the exponent above.
    We have
    \begin{align*}
    2\,\left(1-\frac{4}{8-7\epsilon/4}\right)+\frac{4}{5(8-7\epsilon/8)}&=
        2 - \frac{4}{5}\left(\frac{10}{8 - 7\epsilon/4} - \frac{1}{8 - 7\epsilon/8}\right) \\&= 2 - \frac{4\,\left(10(8 - 7\epsilon/8) - (8 - 7\epsilon/4)\right)}{5(8 - 7\epsilon/4)(8 - 7\epsilon/8)} \\
        &= 2 - \frac{4}{5}\left(\frac{72 - 7\eps}{64 - 21\eps + 49\eps^2/32}\right) \\
        &\geq 1,
    \end{align*}
    for $\epsilon$ sufficiently small.
    Therefore, $k \leq d_i^{(s+t)/5}$ completing the proof of \ref{item: k not too large thm proof}.

    We also have
    \[\frac{\ell_i}{d_i} \geq \frac{\ell_0}{d_0} = \frac{C\,\eta}{\mu}.\]
    Since $\mu < 1$ for $\epsilon$ sufficiently small, this shows 
    \[\ell_i \geq \frac{C\,\eta}{\mu}d_i \geq C\,\eta d_i \geq 4\eta d_i\]
    proving \ref{item: ell bounded by d} as the upper bound always holds for $i < i^*$.

    As $\log(x)/\log\log(x)$ is increasing for $x$ large enough, it follows we can take $\tilde d$ sufficiently large such that
    \[st \,\leq\, \frac{\alpha \epsilon \log d}{\log \log d} \,\leq\, \frac{\alpha\epsilon\log\left(d_i^{40/\epsilon}\right)}{\log\log\left(d_i^{40/\epsilon}\right)} \,\leq\, \frac{40\alpha\log d_i}{\log\log d_i} \,\leq\, \frac{\tilde{\alpha}\log d_i}{\log\log d_i},\]
    when $\alpha \leq \tilde{\alpha}/40$. This proves \ref{item: s and t bounded}.
    
    Again using the fact that $dk^{-1/(s+t)} \leq 2d$, we see for $d$ large enough that
    \[\frac{1}{\log^5(d_i)} \,\leq\, \frac{1}{\log^5(d^{\epsilon/40})} \,=\, \frac{(40/\epsilon)^5}{\log^5(d)} < \frac{\mu}{\log(2d)} \,\leq\, \eta.\]
    Furthermore, for $\epsilon$ sufficiently small, we have $\mu < 1$. Thus, as $d_i \leq d$, we may conclude that
    \[\eta \,\leq\, \frac{1}{\log\left(dk^{-1/(s+t)}\right)} \,\leq\, \frac{1}{\log\left(d_ik^{-1/(s+t)}\right)}.\]
    This shows \ref{item: eta small}.
    Let $0\leq i' \leq i^*-1$ be the largest integer such that
    $\beta_{i'} = d_{i'}^{-1/(200st)}$.
    Then 
    \begin{align*}
        \beta_{i^*-1} &= (1 + 36\eta)^{i^*-1 - i'} d_{i'}^{-1/ (200st)} \\
        &\leq (1+36\eta)^{i^*}\left(d^{\epsilon/40}\right)^{-1/ (200st)} \\
        &\leq \exp\left(36\eta i^*\right)d^{-\epsilon/(800st)}.
    \end{align*}
    By Lemma~\ref{lemma: iteration helper 1}~item~\ref{item: existence of i star}, we may conclude the following:
    \begin{align*}
        \beta_{i^*-1} & \leq \exp\left(36\eta i^*\right)d^{-\epsilon/(800st)}\\
        & \leq \exp\left(576\,\frac{\eta}{\mu}\log \left(dk^{-1/(s+t)}\right)\log\log \left(dk^{-1/(s+t)}\right) \right)d^{-\epsilon/(800st)}\\
        & = \left(\log\left(dk^{-1/(s+t)}\right)\right)^{576}d^{-\epsilon/(800st)}\\
       & \leq \left(\log(2d)\right)^{576}d^{-\epsilon/(800st)}\\
       & \leq \left(\log(2d)\right)^{576}\,d^{-\dfrac{\epsilon\log\log d}{800\alpha\epsilon\log d}}\\
       & \leq \log(d)^{1000}\log(d)^{-\frac{1}{800\alpha}}\\
       & = \log(d)^{1000-\frac{1}{800\alpha}}\\
       & \leq \frac{1}{10},
    \end{align*}
    for $d$ large enough and $\alpha$ small enough. This shows \ref{item: beta small} and completes the proof.
\end{proof}

We are now ready to prove the main theorem:
\begin{theo*}[Restatement of Theorem~\ref{theo:main_theo}]
    There exists a constant $\alpha > 0$ such that for every $\epsilon > 0$, there is $d^* \in \N$ such that the following holds. 
    Suppose that $d$, $s$, $t \in \N$, and $k \in \R$ satisfy
    \[
        d \geq d^*,\quad 1 \,\leq\, t \,\leq\, s, \quad st \leq \frac{\alpha\,\epsilon\,\log d}{\log \log d}, \quad \text{and} \quad 1/2 \,\leq\, k \,\leq\, d^{(s+t)/10}.
    \]
    Define $C$ as follows:
    \[C \defeq \left\{\begin{array}{cc}
          4 + \eps & k \,\leq\, d^{\eps(s+t)/200}, \\
          8 & \text{otherwise.}
        \end{array}\right.\]
    If $G$ is a graph and $\mathcal{H} = (L,H)$ is a correspondence cover of $G$ such that:
    \begin{enumerate}[label=\ep{\normalfont\roman*}]
        \item $H$ is $(k, K_{s,t})$-locally-sparse,
        \item $\Delta(H) \leq d$, and
        \item $|L(v)| \geq C\,d/\log \left(dk^{-1/(s+t)}\right)$ for all $v \in V(G)$.
    \end{enumerate}
    Then, $G$ admits a proper $\mathcal{H}$-coloring.
\end{theo*}

\begin{proof}

Let $\alpha$ be as defined in Lemma \ref{lemma: iteration helper 2}. Let $\epsilon > 0$ be sufficiently small. We will not explicitly compute $d^*$ but simply take $d$ large enough when needed. Thus let $d$ be sufficiently large so that we may apply Lemma \ref{lemma: iteration helper 2}. 

Let $G$ be a graph and $\mathcal{H} = (L, H)$ be a correspondence cover of $G$ satisfying the hypotheses of the theorem. Set
    \begin{align*}
        G_0 \defeq G, \qquad \mathcal{H}_0 = (L_0, H_0) \defeq \mathcal{H}
    \end{align*}

By slightly modifying $\epsilon$ if necessary, and taking $d$ sufficiently large, we may assume $\ell_0$ is an integer. By removing some of the vertices from $H$ if necessary, we may assume that $|L(v)| = \ell_0$ for all $v \in V(G)$.

At this point, we recursively define a sequence of graphs $G_i$ with correspondence cover $\mathcal{H}_i = (L_i, H_i)$ for $0 \leq i \leq i^*$ by iteratively applying Lemma \ref{lemma:iteration} as follows: Lemma \ref{lemma: iteration helper 2} shows conditions \ref{item:d_large_3.1}--\ref{item:eta} of Lemma \ref{lemma:iteration} (with $d_i$ in place of $d$, $\ell_i$ in place of $\ell$) hold for each $0 \leq i \leq i^*-1$. 
Therefore, if $G_i$ is a graph with correspondence cover $\mathcal{H}_i = (L_i, H_i)$ such that for $\beta_i$ as defined at the start of the section,
\begin{enumerate}[label=\ep{\normalfont S\arabic*}]
    \item\label{item:6} $H_i$ is $(k, K_{s,t})$-locally-sparse,
    \item\label{item:7} $\Delta(H_i) \leq 2d_i$,
    \item\label{item:8} the list sizes are roughly between $\ell_i/2$ and $\ell_i$: \[(1-\beta_i)\ell_i/2 \,\leq\, |L_i(v)| \,\leq\, (1+\beta_i)\ell_i \quad \text{for all } v \in V(G_i),\]
    \item\label{item:9} average color-degrees are smaller for vertices with smaller lists of colors: \[\overline{\deg}_{\mathcal{H}_i}(v) \,\leq\, \left(2 - (1 - \beta_i)\frac{\ell_i}{|L_i(v)|}\right)d_i \quad \text{for all } v\in V(G_i),\]
\end{enumerate}
then $G_i$, and $\mathcal{H}_i$ satisfy the required hypotheses of Lemma~\ref{lemma:iteration}.
In particular, there exists a partial $\mathcal{H}_i$-coloring $\phi_i$ of $G_i$ and an assignment of subsets $L_{i+1}(v) \subseteq (L_i)_{\phi_i}(v)$ to each vertex $v \in V(G_i) \setminus \dom(\phi_i)$ such that, setting
\[
    G_{i+1} \defeq G_i[V(G_i) \setminus \dom(\phi_i)], \qquad H_{i+1} \defeq H_i \left[\textstyle\bigcup_{v \in V(G_{i+1})} L_{i+1}(v)\right],
\]
\[
    \text{and} \qquad \mathcal{H}_{i+1} \defeq (L_{i+1}, H_{i+1}),
\]
we have items \ref{item:6}--\ref{item:9} hold for $G_{i+1}$ and $\mathcal H_{i+1}$ with parameter $\beta_{i+1}$.
Thus we may iteratively apply Lemma \ref{lemma:iteration} $i^*$ times, starting with $G_0$ and $\mathcal{H}_0$, to define $G_i$ and $\mathcal{H}_i$ for $0 \leq i \leq i^*$.

We now show $G_{i^*}$ and $\mathcal{H}_{i^*}$ satisfy the hypotheses of Proposition \ref{prop:final_blow}.
By \ref{item:8}, it follows for each $v \in V(G_{i^*})$,
\[|L_{i^*}(v)| \geq (1-\beta_{i^*})\ell_{i^*}/2\geq (1-(1+36\eta)\beta_{i^*-1})\ell_{i^*}/2.\]

Since $\beta_{i^*-1} \leq \frac{1}{10}$ and $\eta \leq \frac{1}{100}$ for $d$ large enough, it follows 
\[(1-(1+36\eta)\beta_{i^*-1})\ell_{i^*}/2 \,\geq\, \left(1-\frac{1+36/100}{10}\right)\ell_{i^*}/2 \,\geq\, \frac{1}{4}\ell_{i^*} \]
for $d$ large enough.
As $\ell_{i^*} \geq 100d_{i^*}$, we have:
\[\frac{1}{4}\ell_{i^*} \geq 25d_{i^*}.\]
We have by \ref{item:7} that $\Delta(H_{i^*}) \leq 2d_{i^*}$, thus for all $c \in V(H_{i^*})$, we have 
\[ 25d_{i^*} \geq 10\deg_{H_{i^*}}(c).\]
Putting this chain of inequalities together yields $|L_{i^*}(v)| \geq 8 \deg_{H_{i^*}}(c)$, as desired. 

Therefore by Proposition \ref{prop:final_blow} there exists an $\mathcal{H}_{i^*}$-coloring of $G_{i^*}$, say $\phi_{i^*}$.
Letting $\phi_i$ be  the partial $\mathcal{H}_{i}$-colorings of $G_i$ for $0 \leq i \leq i^*-1$ guaranteed at each application of Lemma \ref{lemma:iteration}, it follows \[\bigcup_{i = 0}^{i^*}\phi_i\]
is a proper $\mathcal{H}$-coloring of $G$, as desired.
\end{proof}

\section{Concluding remarks}\label{subsec:open_probs}

We conclude with a description of potential future directions of inquiry.
In this paper, we show an asymptotically improved bound on the (correspondence) chromatic number of a graph $G$ of maximum degree $\Delta$ when $G$ has at most $k \ll \Delta^{s+t}$ copies of $K_{s, t}$ in the neighborhood of any vertex. 
It is natural to ask if we can extend the range of $k$ to $k = \Theta\left(\Delta^{s+t}\right)$.
We note that for $k = \binom{\Delta}{s+t}\binom{s+t}{t}$, in the worst case the graph is complete, implying 
\[\chi_c(G) = \Delta + 1 = O\left(\frac{\Delta}{\log\left(\Delta k^{-1/(s+t)}\right)}\right).\]
Therefore, we suspect our results should extend to all values of $k$ in this range (although the proof would require some new ideas; recall the discussion at the end of \S\ref{sec: proof overview}).
Furthermore, in light of our results and Conjecture~\ref{conj:AKS}, we make the following more general conjecture (a correspondence coloring version of Conjecture~\ref{conj:AKS} appeared in \cite{anderson2023colouring}).

\begin{conj}\label{conj:aks_loc_sparse}
    For every graph $F$, the following holds for $\Delta \in \N$ large enough.
    Let 
    \[1/2 \,\leq\, k \,<\, \Delta^{|V(F)|},\]
    and let $G$ be a $(k,F)$-locally-sparse graph of maximum degree $\Delta$.
    Then, 
    \[\chi_c(G) = O\left(\frac{\Delta}{\log\left(\Delta k^{-1/|V(F)|}\right)}\right),\] 
    where the constant factor in the $O(\cdot)$ may depend on $F$.
\end{conj}

We note that a simple application of Theorem \ref{theo:DKPSCorrespondence} together with Proposition~\ref{prop:KST_multiple} proves Conjecture \ref{conj:aks_loc_sparse} for bipartite graphs $F$ and the following range of $k$:
\[|V(F)|^2\Delta^{|V(F)| - 1} \,\leq\, k \,\leq\, \frac{\Delta^{|V(F)|}}{(\log \Delta)^c},\]
for some constant $c \defeq c(F) > 0$.
As a result of Corollary~\ref{corl:DP_chromatic_number}, it remains to consider the range
\[ \Delta^{|V(F)|/10}\,<\, k \,<\, |V(F)|^2\Delta^{|V(F)| - 1}, \qquad \text{and} \qquad k \,>\, \frac{\Delta^{|V(F)|}}{(\log \Delta)^c},\]
in order to verify the conjecture for bipartite graphs.

We remark that the constant $C$ in Corollary \ref{corl:DP_chromatic_number} is independent of $F$. 
This is in line with several recent results surrounding Conjecture \ref{conj:AKS} which prove the constant in the $O(\cdot)$ is independent of $F$ (such as \cite{PS15, Molloy} for $F = K_3$, \cite{DKPS} for fans and cycles, \cite{anderson2023colouring} for $F$ bipartite, and \cite{anderson2022coloring} for $F$ almost-bipartite; see Table~\ref{table:bounds}). 
To this end, we conjecture the following, which is a strengthening of \cite[Conjecture~1.7]{anderson2022coloring}:

\begin{conj}\label{conj:aks_loc_sparse universal constant}
    There exists a universal constant $C \in \R$ such that for every graph $F$, the following holds for $\Delta \in \N$ large enough.
    Let 
    \[1/2 \,\leq\, k \,<\, \Delta^{|V(F)|},\]
    and let $G$ be a $(k,F)$-locally-sparse graph of maximum degree $\Delta$.
    Then, 
    \[\chi_c(G) \leq \frac{C\,\Delta}{\log\left(\Delta k^{-1/|V(F)|}\right)}.\] 
\end{conj}

We note that for an arbitrary graph, a simple counting argument shows that
\[k \leq \binom{\Delta}{|V(F)|}\,\frac{|V(F)|!}{|\mathsf{Aut}(F)|},\]
where the bound is tight when considering the number of copies of $F$ in $K_\Delta$.
As the above is strictly less than $\Delta^{|V(F)|}$, the bounds on $k$ in Conjecture~\ref{conj:aks_loc_sparse} and Conjecture~\ref{conj:aks_loc_sparse universal constant} cover all possible cases.

We also note that Corollary~\ref{corl:DP_chromatic_number} verifies this conjecture in the case when $F$ is bipartite for a smaller range of $k$ than stated. 
Additionally, we remark that Conjecture~\ref{conj:aks_loc_sparse universal constant} is quite challenging. 
Indeed, Conjecture~\ref{conj:AKS} itself is considered ambitious (see the discussion at the end of \cite[\S1.1]{anderson2022coloring}).
As such, falsifying Conjecture~\ref{conj:aks_loc_sparse universal constant} may be a much more feasible task than disproving Conjecture~\ref{conj:AKS} and so we pose it as an interesting question.

As mentioned earlier, the best known upper bound for Conjecture~\ref{conj:AKS} and arbitrary $F$ is $O\left(\Delta\log\log\Delta/\log\Delta\right)$ due to Johansson \cite{Joh_sparse}.
In fact, his proof holds in the list coloring setting as well.
The constant factor was improved by Molloy \cite{Molloy}, and Bernshteyn extended the result to correspondence coloring \cite{bernshteyn2019johansson}.
A natural question to ask is whether we can prove a similar bound for $(k, F)$-locally-sparse graphs.
This would be the first step toward proving Conjecture~\ref{conj:aks_loc_sparse} and Conjecture~\ref{conj:aks_loc_sparse universal constant}.

\subsection*{Acknowledgements}
We thank J\'ozsef Balogh, Peter Bradshaw, and Ethan White for pointing out an error in an earlier version of this manuscript, and Anton Bernshteyn for his helpful comments.
We also thank the anonymous referee for their helpful suggestions.

\vspace{0.1in}
\printbibliography
\vspace{0.1in}

\appendix
\section{Proof of Proposition~\ref{prop: G F free implies H F free}}\label{sec:appendix}

We restate the proposition here for convenience.

\begin{prop*}[\ref{prop: G F free implies H F free}]
    Let $F$ and $G$ be graphs and let $\mathcal{H} = (L, H)$ be a correspondence cover of $G$.
    For $k \in \R$, the following holds:
    \begin{enumerate}[label=\ep{\normalfont S\arabic*}]
        \item\label{item:local} If $G$ is $(k, F)$-locally-sparse, then so is $H$.
        
        \item\label{item:global} If $F$ satisfies
        \[\forall\, u, v \in V(F),\quad uv \notin E(F) \implies N_F(u) \cap N_F(v) \neq \0,\]
        then if $G$ is $F$-free, $H$ is as well.
    \end{enumerate}
\end{prop*}

\begin{proof}
    We first prove \ref{item:local}.
    Let $v \in V(G)$ and $c \in L(v)$. From \ref{dp:matching} it follows  $|N_H(c) \cap L(u)| \leq 1$ for each $u \in V(G)$.
    Suppose $H[N_H(c)]$ contains at least $\lfloor k\rfloor + 1$ copies of $F$.
    Let $X$ be one such copy.
    By the earlier observation, the map $c' \to L^{-1}(c')$ is injective for $c' \in X$, i.e., no two vertices of $X$ are in $L(u)$ for any $u \in V(G)$. Thus the map
    $\psi : V(X) \to V(G)$ defined by $\psi(c') = L^{-1}(c')$ is an isomorphism from $X$ to a subgraph of $G$. 
    In addition, as $V(X) \subseteq N_H(c)$, it follows by the second clause in \ref{dp:matching} that $L^{-1}(c') \in N_G(v)$ for each $c' \in V(X)$. Therefore, $\psi(X) \subseteq G[N_G(v)]$. Thus if $H$ contains a copy of $F$ in a neighborhood, then so does $G$.

    Now, let $X\neq X'$ be distinct copies of $F$ in $H[N_H(c)]$.
    As shown earlier, there are subgraphs $\psi(X),\,\psi(X')$ in $G[N_G(v)]$ isomorphic to $X,\,X'$ respectively. It is now enough to show that $\psi(X)\neq \psi(X')$.
    If $V(X) \neq V(X')$, then by the fact that $c \to L^{-1}(c)$ is injective for $c \in X$, we have $V(\psi(X)) \neq V(\psi(X'))$ and thus $\psi(X)\neq \psi(X')$.
    If $V(X) = V(X')$, then as $X \neq X'$, it follows there is some edge $c_1c_2 \in E(X) \setminus E(X')$ .
    Since $\psi$ is an isomorphism, the edge $L^{-1}(c_1)L^{-1}(c_2) \in E(\psi(X)) \setminus E(\psi(X'))$, and $\psi(X)\neq \psi(X')$. Thus $\psi(X)$ and $\psi(X')$ are two different copies of $F$ in $G[N_G(v)]$. It follows that if $H$ contains $\lfloor k\rfloor + 1$ copies of $F$ in $H[N_H(c)]$, then $G$ contains $\lfloor k\rfloor + 1$ copies of $F$ in $G[N_G(v)]$, contradicting the local sparsity $G$. Thus $H$ is $(k, F)$-locally-sparse.

    We now prove \ref{item:global}. Suppose $X$ is a copy of $F$ in $H$. 
    From the argument in the previous paragraph, if the map $\phi: V(X) \to V(G)$ by $c \to L^{-1}(c)$ is injective, then it follows that $\phi(X)$ is a copy of $F$ in $G$.
    If $F$ is a clique, and $X$ is a copy of $F$, then clearly \ref{dp:list_independent} implies $\phi$ is injective.
    Suppose $F$ is not a clique and let $X$ be a copy of $F$ such that $\phi$ is not injective, i.e., there exist $c, c' \in V(X)$ with $c, c' \in L(u)$ for some $u \in V(G)$. Then, as \ref{dp:list_independent} states $L(u)$ is an independent set, it follows $cc' \notin E(F)$. Thus, by the assumption of \ref{item:global}, it follows $N_X(c) \cap N_X(c') \neq \0$. However, this violates \ref{dp:matching}. Thus $\phi$ is injective.
    Therefore, if $H$ contains a copy of $F$, then so does $G$.
    As $G$ is $F$-free, we conclude $H$ is as well.
\end{proof}

\section{Proof of Proposition~\ref{prop:KST_multiple}}\label{section: proof of KST multiple}

Let us first restate the result.

\begin{prop*}[Restatement of Proposition~\ref{prop:KST_multiple}]
    Let $G$ be an $n$-vertex $(k, K_{s, t})$-sparse graph.
    For $k^\star = \frac{(s+t)^2\,n^{s+t - 1}}{2\,s^s\,t^t}$, we have
    \[|E(G)| \leq \left\{\begin{array}{cc}
        4\,n^{2-1/(st)} & \text{ if } k \leq k^\star;\vspace{5pt} \\
        2\,s^{1/t}\,t^{1/s}\,k^{1/st}\,n^{2-1/s - 1/t} & \text{ if } k > k^\star.
    \end{array}\right.\]
\end{prop*}

We will use the following result of Alon in our proof.

\begin{Lemma}[{Contrapositive of \cite[Lemma 2.1]{alon2002testing}}]\label{lemma: alon}
    Let $s \geq t \geq 1$ and let $G = (V, E)$ be a graph such that there are at most $k$ homomorphisms from a labelled copy of $H = K_{s, t}$ into $G$.
    Then, $|E| \leq k^{1/st}\,n^{2 - 1/s - 1/t}/2$.
\end{Lemma}

\begin{proof}[Proof of Proposition~\ref{prop:KST_multiple}]
    It is easy to see that there are at most $(s+t)^2\,n^{s+t - 1}$ non-injective homomorphisms from a labelled copy of $K_{s, t}$ into $G$.
    As $G$ contains at most $k$ copies of $K_{s, t}$, the number of injective homomorphisms from a labelled copy of $K_{s, t}$ into $G$ is at most
    \[2^{\mathbbm{1}\set{s = t}}\,s!\,t!\,k \,\leq\, 2\,s^s\,t^t\,k.\]
    Therefore, the number of homomorphisms from $K_{s, t}$ into $G$ is at most 
    \[(s+t)^2\,n^{s+t - 1} + 2\,s^s\,t^t\,k \,\leq\, \left\{\begin{array}{cc}
        2\,(s+t)^2\,n^{s+t - 1} &  \text{ if } k \leq k^\star; \\
        4\,s^s\,t^t\,k &  \text{ if } k > k^\star.
    \end{array}\right.\]
    As $s, t \in \N$ and $s \geq t \geq 1$, we have
    \[\left(2(s+t)^2\right)^{1/(st)} \leq 8, \quad \text{and} \quad 4^{1/(st)} \leq 4.\]
    The claim now follows by applying Lemma~\ref{lemma: alon}.
\end{proof}

\end{document}